\documentclass{article}
\usepackage{amsthm}
\newcommand*{\inst}[1]{}
\date{}

\usepackage{xcolor}
\usepackage{tikz}
\usepackage{graphicx}

\newcommand*{\keywords}[1]{Keywords: #1}

\usepackage{amsmath,amssymb,amsfonts}
\usepackage{url}

\usepackage{booktabs}
\usepackage{multirow}

\usepackage[capitalise]{cleveref}

\theoremstyle{definition}
\newtheorem{theorem}{Theorem}[section]

\newtheorem{lemma}[theorem]{Lemma}
\newtheorem{corollary}[theorem]{Corollary}
\newtheorem{example}[theorem]{Example}

\usepackage{float}

\newcommand*{\s}[1]{\ensuremath{\mathsf{#1}}}
\newcommand*{\Define}[1]{{\color{blue}\emph{#1}}}
\newcommand*{\SA}{\textsf{\textup{SA}}}
\newcommand*{\BWT}{\textsf{\textup{BWT}}}
\newcommand*{\ISA}{\textsf{\textup{ISA}}}
\newcommand*{\word}{\ensuremath{\s{T}}}
\newcommand*{\Word}{\ensuremath{\s{V}}}
\newcommand*{\bmodz}{\mathbin{\operatorname{mod}_0}}

\newcommand*{\functionname}[1]{{\ensuremath{\renewcommand{\rmdefault}{ptm}\fontfamily{ppl}\selectfont\textrm{\textup{#1}}}}} \newcommand*{\BalancedTwo}{\ensuremath{\functionname{balanced}_2}}

\begin{document}

\title{On Arithmetically Progressed Suffix Arrays and related Burrows--Wheeler Transforms\thanks{This paper is an extension of a contribution to the Prague Stringology Conference 2020, published in~\cite{DKKS20}.}}
\author{Jacqueline W. Daykin\inst{1} 
\and Dominik K\"oppl\inst{2} 
\and David K\"ubel\inst{3} 
\and Florian Stober\inst{4}}

\maketitle
\begin{abstract}
We characterize those strings whose suffix arrays are based on arithmetic progressions, in particular, arithmetically progressed permutations where all pairs of successive entries of the permutation have the same difference modulo the respective string length.
We show that an arithmetically progressed permutation~$P$ coincides with the suffix array of a unary, binary, or ternary string.
We further analyze the conditions of a given~$P$ under which we can find a uniquely defined string over either a binary or ternary alphabet having $P$ as its suffix array.
For the binary case, we show its connection to lower Christoffel words, balanced words, and Fibonacci words.
In addition to solving the
arithmetically progressed suffix array problem, we give the shape of the Burrows--Wheeler transform of those strings solving this problem.
These results give rise to numerous future research directions.
\end{abstract}


\section{Introduction}

The integral relationship between the suffix array~\cite{manber93sa} (SA) and Burrows--Wheeler transform~\cite{burrows94bwt} (BWT) is explored by Adjeroh et al.~\cite{ABM-08}, who also illustrate the versatility of the BWT beyond its original motivation in lossless block compression~\cite{burrows94bwt}.
BWT applications include compressed index structures using backward search pattern matching, multimedia information retrieval, bioinformatics sequence processing, and it is at the heart of the bzip2 suite of text compressors.
By its association with the BWT, this also indicates the importance of the SA data structure and hence our interest in exploring its combinatorial properties.

These combinatorial properties can be useful when 
checking the performance or integrity of string algorithms or data structures on string sequences in testbeds
when properties of the employed string sequences are well understood.
In particular, due to current trends involving massive data sets, 
indexing data structures need to work in external memory (e.g.,~\cite{bingmann16esais}), 
or on distributed systems (e.g.~\cite{fischer19distributed}).
For devising a solution adaptable to these scenarios, 
it is crucial to test whether the computed index (consisting of the suffix array or the BWT, for instance) is correctly stored on the hard disk or on the computing nodes, respectively.
This test is more cumbersome than in the case of a single machine working only with its own RAM\@.
One way to test is to compute the index for an instance, whose index shape can be easily verified.
For example, one could check the validity of the computed BWT
on a Fibonacci word since the shape of its BWT is known~\cite{mantaci03sturmian,christodoulakis06fibonacci,simpson08bwt}.

Other studies based on Fibonacci words are the suffix tree~\cite{rytter06fibonacci} or the Lempel-Ziv 77 (LZ77) factorization~\cite{berstel06strumian}.
In~\cite{koppl15fibonacci}, the suffix array and its inverse of each even Fibonacci word is studied as an arithmetic progression.
In this study, the authors, like many at that time, did not append the artificial $\texttt{\$}$ delimiter (also known as a sentinel) to the input string, thus allowing suffixes to be prefixes of other suffixes.
This small fact makes the definition of $\BWT_{\word}[i] = \word[\SA_{\word}[i]-1]$ for a string~\word{} with suffix array~$\SA_{\word}$ 
incompatible with the traditional BWT defined on the BWT matrix, namely the lexicographic sorting of all cyclic rotations of the string~\word{}.

For instance, $\SA_{\texttt{bab}} = [2,3,1]$ with $\BWT_{\texttt{bab}} = \texttt{bab}$, while
$\SA_{\texttt{bab\$}} = [4,2,3,1]$ and $\BWT_{\texttt{bab\$}} = \texttt{bba\$}$ with $\texttt{\$} < \texttt{a} < \texttt{b}$.
However, the traditional BWT constructed by reading the last characters of the lexicographically sorted cyclic rotations $[\texttt{ab\textit{b}}, \texttt{ba\textit{b}}, \texttt{bb\textit{a}}]$ of \texttt{bab} yields \texttt{bba}, 
which is equal to $\BWT_{\texttt{bab\$}} = \texttt{bba\$}$ after removing the \texttt{\$} character. 

Note that not all strings are in the BWT image.
An $O(n \log n)$-time algorithm is given by {Giuliani et al.}~\cite{giuliani19dollar} for
identifying all the positions in a string \s{S} where a $\texttt{\$}$ can be inserted into so that \s{S} becomes the BWT image of a string ending with $\texttt{\$}$.

Despite this incompatibility in the suffix array based definition of the BWT, we can still observe a regularity for even Fibonacci words~\cite[Sect.\ 5]{koppl15fibonacci}.
Similarly, both methods for constructing the BWT are compatible when the string~\word{} consists of a Lyndon word.
The authors of~\cite[Remark 1]{koppl15fibonacci} also observed similar characteristics for other, more peculiar string sequences.
For the general case, the $\texttt{\$}$ delimiter makes both methods equivalent, however the suffix array approach is typically preferred as it requires $O(n)$ time~\cite{nong09sais} compared to $O(n^2)$ with the BWT matrix method~\cite{burrows94bwt}.
By utilizing combinatorial properties of the BWT, an in-place algorithm is given by Crochemore et al.~\cite{CrochemoreGKL15}, which avoids the need for explicit storage for the suffix sort and output arrays, and runs in $O(n^2)$ time using $O(1)$ extra memory (apart from storing the input text).
K\"{o}ppl et al.~\cite[Sect.~5.1]{koppl20inplace} adapted this algorithm to compute the traditional BWT within the same space and time bounds.

Up to now, it has remained unknown whether we can formulate a class of string sequences for which we can give the shape of the suffix array as an arithmetic progression (independent of the $\texttt{\$}$ delimiter).
With this article, we catch up on this question, and establish a correspondence between strings and suffix arrays generated by arithmetic progressions.
Calling a permutation of integers $[1..n]$ \emph{arithmetically progressed} if all pairs of successive entries of the permutation
have the same difference modulo $n$, 
we show that an arithmetically progressed permutation coincides with the suffix array of a unary, binary or ternary string.
We analyze the conditions of a given arithmetically progressed permutation~$P$ under which we can find a uniquely defined string~\word{} over either a unary, a binary, or ternary alphabet having $P$ as its suffix array.

The simplest case is for unary alphabets: 
Given the unary alphabet $\Sigma := \{\texttt{a}\}$ and a string~\word{} of length~$n$ over~$\Sigma$, 
$\SA_{\word} = [n, n-1, \ldots, 1]$ is an arithmetically progressed permutation with ratio~$-1 \equiv n-1 \mod n$.

For the case of a binary alphabet $\{\texttt{a}, \texttt{b}\}$, several strings of length~$n$ exist that solve the problem.
Trivially, the solutions for the unary alphabet also solve the problem for the binary alphabet.
However, studying those strings of length~$n$ whose suffix array is $[n, n-1, \ldots, 1]$, there are now multiple solutions: 
each $\word = \texttt{b}^r \texttt{a}^s$ with $r, s \in [0..n]$ such that $r+s=n$ has this suffix array. 
Similarly, $\word=\texttt{a}^{n-1}\texttt{b}$ has the suffix array $\SA_{\word} = [1, 2, \ldots, n]$, which is an arithmetically progressed permutation with ratio 1.

\newcommand*{\precV}{\ensuremath{\prec_{\textup{V}}}}

  A non-lexicographic ordering on strings, \Define{the $V$-order}, considered for a modified FM-index~\cite{alatabbi19vorder}, provides a curious example for the case with ratio $-1$:
	if \s{S} is a proper subsequence of a string \word{} of length~$n$,
	then \s{S} precedes \word{} in $V$-order, written $\s{S} \precV \word$.
	This implies that $\word[n] \precV \word[n-1..n] \precV \cdots \precV \word[1..n]$, so that $\SA_{\word}[i] = n - i + 1$ for every~$i \in [1..n]$, thus enabling trivial suffix sorting.

	Clearly, any string ordering method that prioritizes minimal length within its definition will have a suffix array progression ratio of $-1$.
	Binary Gray codes\footnote{Originally known as \emph{Reflected Binary Code} by the inventor Frank Gray.} have this property and are ordered so that adjacent strings differ by exactly one bit~\cite{Gray53}.
	These codes, which exhibit non-lexicographic order, have numerous applications, notably in error detection schemes such as flagging unexpected changes in data for digital communication systems, and logic circuit minimization.
	While it seems that Gray code order has not been applied directly to order the rows in a BWT matrix, just four years after the BWT appeared in 1994~\cite{burrows94bwt}, Chapin and Tate applied the concept when they investigated the effect of both alphabet and string reordering on BWT-based compressibility~\cite{CT98}.
	Their string sorting technique for the BWT matrix
ordered the strings in a manner analogous to reflected Gray codes, but for more general alphabets.
	This modification inverted the sorting order for alternating character positions with which they demonstrated improved compression.

	To date the only known \BWT~designed specifically for binary strings is the binary block order $B$-BWT~\cite{DGGLLLP16}.
	This binary string sorting method prioritizes length, thus also exhibiting a suffix array with progression ratio $-1$.
	Ordering strings of the same length, as applicable to forming the BWT matrix, is by a play on decreasing and increasing lexicographic ordering of the run length exponents of blocks of bits.
Experimentation showed roughly equal compressibility to the original \BWT~in binary lexicographic order but with instances where the $B$-BWT gave better compression.

	To see that the above two binary orders are distinct: 1101 comes before 1110  in Gray code order, whereas, 1110 comes before 1101 in binary block order.

In what follows, we present a comprehensive analysis of strings whose suffix arrays are arithmetically progressed permutations (under the standard lexicographic order).
In practice, such knowledge can reduce the $O(n)$ space for the suffix array to $O(1)$.

The structure of the paper is as follows.\footnote{Compared to the conference version, we added applications to Christoffel words and balanced words, generalized our results for larger alphabets, and gave examples for indeterminate strings.}
In Section~\ref{secPrelims} we give the basic definitions and background, and also deal with the elementary case of a unary alphabet.
We present the main results in Section~\ref{secAPSA}, where we cover ternary and binary alphabets, and consider inverse permutations. 
In \cref{secApplications}, we illustrate the binary case for Christoffel words and in particular establish that every (lower) Christoffel word has an arithmetically progressed suffix array. We go on to link the binary characterization to balanced words and Fibonacci words.
A characterization of strings with larger alphabets follows in \cref{secLargerAlphabets}.
We overview the theme concepts for meta strings in Section~\ref{secMeta}.
We conclude in Section~\ref{secConclusion} and propose a list of open problems and research directions, showing there is plenty of scope for further investigation.
But for all that, we proceed to the foundational concepts presented in the following section, starting with the case of a unary alphabet.

\section{Preliminaries}
\label{secPrelims}

Let $\Sigma$ be an alphabet with size $\sigma := |\Sigma|$.
An element of $\Sigma$ is called a \Define{character}\footnote{Also known as \emph{letter} or \emph{symbol} in the literature.}.
Let $\Sigma^+$ denote the set of all nonempty finite strings over $\Sigma$.
The \Define{empty string} of length zero is
denoted by \s{\varepsilon}; we write $\Sigma^* = \Sigma^+ \cup \{\s{\varepsilon}\}$.
Given an integer $n \ge 1$,
a \Define{string}\footnote{Also known as \emph{word} in the
literature.} 
of length~$n$ over $\Sigma$ takes the form $\word = t_1 \cdots t_n$ with each $t_i \in \Sigma$.
We write $\word = \word[1..n]$ with $\word[i] = t_i$.
The length~$n$ of a string \word{} is denoted by $|\word|$.
If $\word = \s{uwv}$ for some strings $\s{u,w,v} \in \Sigma^*$, then \s{u} is a \Define{prefix}, \s{w}
is a \Define{substring}, and \s{v} is a \Define{suffix} of \word{}; 
we say $\s{u}$ (resp.\ \s{w} and \s{v}) is \Define{proper} if $\s{u} \neq \word$ (resp.\ $\s{w} \neq \word$ and $\s{v} \neq \word$).
We say that a string~\word{} of length~$n$ has \Define{period}~$p \in [1..n-1]$ if $\word[i] = \word[i+p]$ for every~$i \in [1..n-p]$ 
(note that we allow periods larger than $n/2$).
If $\word = \s{uv}$,
then \s{vu} is said to be a \Define{cyclic rotation} of \word{}.
 A string \word{} is said to be a \Define{repetition} if and only if it has a factorization
 $\word{} = \word^k$ for some integer $k \ge 1$;
 otherwise, \word{} is said to be \Define{primitive}.
A string that is both a proper prefix and a proper suffix of a 
string $\word \neq  \s{\varepsilon}$ is called a \Define{border} of~\word{}; a string is 
\Define{border-free} if the only border it has is the empty string \s{\varepsilon}.

If $\Sigma$ is a totally ordered alphabet with order~$<$, 
then this order~$<$ induces the \Define{lexicographic ordering} $\prec$ on $\Sigma^*$ such that $\s{u} \prec \s{v}$ for two strings $\s{u},\s{v} \in \Sigma^*$
if and only if either \s{u} is a proper prefix of
\s{v}, or $\s{u}=\s{r}a\s{s}$, $\s{v}=\s{r}b\s{t}$ for two characters $a,b \in \Sigma$ such that $a<b$ 
and for some strings $\s{r},\s{s},\s{t} \in \Sigma^*$.
In the following, we select a totally ordered alphabet~$\Sigma$ having three characters $\texttt{a}, \texttt{b}, \texttt{c}$ with $\texttt{a} < \texttt{b} < \texttt{c}$.

A string~\word{} is a \Define{Lyndon word} if it is strictly least in the lexicographic order among all its cyclic rotations~\cite{lothaire97combinatorics}.
For instance, \texttt{abcac} and \texttt{aaacbaabaaacc} are Lyndon words, 
while the string~\texttt{aaacbaabaaac} with border~\texttt{aaac} is not.     

	A reciprocal relationship exists between the suffix array of a text and its Lyndon factorization, that is, 
	the unique factorization of the text by greedily choosing the maximal length Lyndon prefix while processing the text from start to end: 
	the Lyndon factorization of a text can be obtained from its suffix array~\cite{HR03}; conversely, the suffix array of a text can be constructed iteratively from its Lyndon factorization~\cite{MRRS14}.

	Lyndon words have numerous applications in combinatorics and algebra, and prove challenging entities due to their non-commutativity.
	Additionally, Lyndon words can arise naturally in Big Data --  an instance in a biological sequence over the DNA alphabet, 
	with $\Sigma = \{\texttt{A} < \texttt{C} < \texttt{G} <\texttt{T}\}$, is the following substring occurring in a SARS-CoV-2 genome\footnote {\url{https://www.ncbi.nlm.nih.gov/nuccore/NC_045512.2?report=fasta}}:
	\begin{center}
	\texttt{AAAAACAGTAAAGTACAAATAGGAGAGTACACCTTTGAAAAAGGTGACTATGGTGAT}
	\end{center}

For the rest of the article, we take a string~\word{} of length~$n \ge 2$.
The suffix array $\SA := \SA_{\word}[1..n]$ of~\word{}
is a permutation of the integers $[1..n]$ such that $\word[\SA[i]..n]$ is the $i$-th lexicographically smallest suffix of~\word{}.
We denote with \ISA{} its inverse, i.e., $\ISA[\SA[i]] = i$.
By definition, \ISA{} is also a permutation.
The string~\BWT{} with $\BWT[i] = \word[\SA[i]-1 \mod n]$ is the (SA-based) BWT  of~\word{}.

The focus of this paper is on arithmetic progressions.
An \Define{arithmetic progression} is a sequence of numbers such that the differences between all two consecutive terms are of the same value:
Given an arithmetic progression~$\{p_i\}_{i \ge 1}$, there is an integer~$k \ge 1$ such that 
$p_{i+1} = p_i + k$ for all $i \ge 1$.
We call $k$ the \Define{ratio} of this arithmetic progression.
Similarly to sequences, we can define permutations that are based on arithmetic progressions:
An \Define{arithmetically progressed permutation} with ratio~$k \in [1..n-1]$
is an array~$P := [p_1,\ldots,p_n]$ 
with $p_{i+1} = p_i + k \mod n$ for all $i \in [1..n]$, where we stipulate that $p_{n+1} := p_1$.\footnote{We can also support negative values of~$k$:
Given a negative~$k < 0$, we exchange it with~$k' := n-k \mod n \in [1..n]$ and use $k'$ instead of~$k$.}
Here $x \bmod n := x$ if $x \le n$, $x - n \mod n$ for an integer~$x \ge 1$, and $x+n \mod n$ for $x < 1$.
In particular, $n \mod n = 0 \mod n = n$.
In what follows, we want to study (a) strings whose suffix arrays are arithmetically progressed permutations, and (b) the shape of these suffix arrays.
For a warm-up, we start with the unary alphabet:

\begin{theorem}\label{thmUnaryAlphabet}
	Given the unary alphabet~$\{\mathtt{a}\}$, the suffix array of a string of length~$n$ over~$\{\mathtt{a}\}$ is uniquely defined by 
	the arithmetically progressed permutation $[n,n-1,\ldots,1]$ with ratio~$n-1$.
\end{theorem}

Conversely, given the arithmetically progressed permutation $P = [n,n-1,\ldots,1]$, 
we want to know the number of strings from a general totally ordered alphabet~$\Sigma = [1..\sigma]$ with the natural order $1<2< \cdots < \sigma$, having $P$ as their suffix array.
For that, we fix a string~\word{} of length~$n$ with $\SA_{\word} = P$.
Let $s_j \ge 0$ be the number of occurrences of the character~$j \in \Sigma$ appearing in~$\word{}$.
Then $\sum_{j=1}^\sigma s_j = n$.
By construction, each character~$j$ has to appear after all characters~$k$ with $k > j$.
Therefore, 
$\word = \sigma^{s_\sigma} (\sigma-1)^{s_{\sigma-1}} \cdots 1^{s_1}$
such that the position of the characters are uniquely determined.\footnote{Note that we slightly misused notation as the exponentiations of the characters being integers have to be understood as writing a character as many times as the exponent. So $1^{s_1}$ does not give $1$ but a string of length~$s_1$ having only $1$'s as characters.} 
In other words, we can reduce this problem to the classic stars and bars problem~\cite[Chp.~II, Sect.~5]{feller68probability}
with $n$ stars and $\sigma - 1$ bars, yielding $n + \sigma - 1 \choose n$ possible strings.
Hence we obtain:

\begin{theorem}\label{thmUnaryEnumeration}
There are $n + \sigma - 1 \choose n$ strings of length~$n$ over an alphabet with size~$\sigma$ having the suffix array~$[n,n-1,\ldots,1]$.
\end{theorem}
As described above, strings of \cref{thmUnaryEnumeration} have the form $\sigma^{s_\sigma} (\sigma-1)^{s_{\sigma-1}} \cdots 1^{s_1}$.
The BWT based on the suffix array is
$1^{s_1-1} 2^{s_2} \cdots \sigma^{s_\sigma} 1$.
For $s_1 \ge 2$, it does not coincide with the BWT based on the rotations since
the lexicographically smallest rotation is $1^{s_1} \sigma^{s_\sigma} \cdots 2^{s_2}$, and hence the first entry of this BWT is~$2$.
For $s_1 = 1$, the last character `1' acts as the dollar sign being unique and least among all characters, making both BWT definitions equivalent.

For the rest of the analysis, we omit the arithmetically progressed permutation $[n, n-1,\allowbreak \ldots,1]$ of ratio $k= n-1$ as this case is complete. 
All other permutations (including those of ratio $k=n-1$) are covered in our following theorems whose results we summarized in \cref{tabOverview}.

\begin{figure}
\centering
\begin{tabular} {c @{\hskip 1em} | c c c @{\hskip 1em} l }
    \toprule
	$p_1$		            & $k$ 	        & Min.
	$\sigma$     & Properties of Strings  & Reference \\
    \midrule
	$1$                     &               & 2                         &  unique, Lyndon word & \cref{thmBinaryStringCharacteristic} \\[1em]
	$k+1$                   &               & 2                         &  unique, period $(n-k)$ & \cref{thmBinaryStringCharacteristic} \\[1em]
	\multirow{2}{*}{$n$}    & $\neq (n-1)$  & 2                         &  unique, period $(n-k)$ & \cref{thmBinaryStringCharacteristic} \\
							& $= (n-1)$     & 1                         &  trivially periodic & \cref{thmUnaryEnumeration} \\[1em]
							$\not\in\lbrace 1, k+1, n\rbrace$   &   & 3                         &  unique & \cref{thmTernaryAlphabet} \\
    \bottomrule
    \end{tabular}
    \caption{Characterization of strings whose suffix array is an arithmetic progression $P = [p_1, \ldots p_n]$ of ratio $k$.
    The choice of $p_1$ determines the minimum size of the alphabet and whether a string is unique, periodic or a Lyndon word.
	The column \emph{Min. $\sigma$} denotes the smallest possible size $\sigma$ for which there exists such a string whose characters are drawn from an alphabet $\Sigma$ with $|\Sigma| = \sigma$.
}
	\label{tabOverview}
\end{figure}

\section{Arithmetically Progressed Suffix Arrays}
\label{secAPSA}

We start with the claim that each arithmetically progressed permutation coincides with the suffix array of 
a string on a ternary alphabet.
Subsequently, given an arithmetically progressed permutation~$P$, we show that either there is precisely one string~\word{} with $\SA_\word = P$ whose characters are drawn from a \emph{ternary} alphabet, or, if there are multiple candidate strings, then there is precisely one whose characters are drawn from a \emph{binary} alphabet.
For this aim, we start with the restriction on~$k$ and~$n$ to be coprime:

\subsection{Coprimality}
Two integers are \Define{coprime}\footnote{Also known as \emph{relatively prime} in the
literature.} 
if their greatest common divisor (gcd) is one.
An \Define{ideal} $k\mathbb{N} := \{ki\}_{i \in \mathbb{N}}$ is a subgroup of $([1..n], +)$.
It \Define{generates} $[1..n]$ if $|k\mathbb{N}| = n$, i.e., $k\mathbb{N}  = [1..n]$.
Fixing one element $P[1] \in k\mathbb{N}$ of an ideal $k\mathbb{N}$ generating $[1..n]$
induces an arithmetically progressed permutation~$P[1..n]$ with ratio~$k$ by setting $P[i+1] \gets P[i] + k$ for every~$i \in [1..n-1]$.
On the contrary, each arithmetically progressed permutation with ratio~$k$ induces an ideal $k\mathbb{N}$ 
(the induced ideals are the same for two arithmetically progressed permutations that are shifted).
Consequently, there is no arithmetically progressed permutation with ratio~$k$ if $k$ and $n$ are not coprime 
since in this case $\{ (ki) \mod n \mid i \ge 1 \} \subsetneq [1..n]$, from which we obtain:

\begin{lemma}
The numbers $k$ and $n$ must be coprime if there exists an arithmetically progressed permutation of length~$n$ with ratio~$k$.
\end{lemma}

\subsection{Ternary Alphabet}

\begin{figure}[tb]

  \setlength{\tabcolsep}{0.25em}
	\centerline{\begin{tabular}{r
@{\hskip 1em}
*{8}{l}
@{\hskip 1em}
r*{8}{c}l
@{\hskip 1em}
ll}
		\toprule
		Rotation & \multicolumn{8}{c}{\word{}} & \multicolumn{10}{c}{P} & \multirow{2}{*}{\begin{minipage}{4em}{$p_1-k-1$ \\ $\bmod ~n$}\end{minipage}} & $\BWT_{\word}$ \\
		\cmidrule(r){2-9}
		\cmidrule{11-18}
		&
		{\scriptsize 1} &
		{\scriptsize 2} &
		{\scriptsize 3} &
		{\scriptsize 4} &
		{\scriptsize 5} &
		{\scriptsize 6} &
		{\scriptsize 7} &
		{\scriptsize 8} &
		  &
		{\scriptsize 1} &
		{\scriptsize 2} &
		{\scriptsize 3} &
		{\scriptsize 4} &
		{\scriptsize 5} &
		{\scriptsize 6} &
		{\scriptsize 7} &
		{\scriptsize 8} &
		  &
		\\
		(1) & 
		\texttt{b} & 
		\texttt{a} & 
		\texttt{b} & 
		\texttt{b} & 
		\texttt{a} & 
		\texttt{b} & 
		\texttt{a} & 
		\texttt{c} & 
	[ &
		5, &
		2, &
		7$\mid$ &
		4, &
		1, &
		6, &
		3$\mid$ &
		8 &
		] &
		7 & $\texttt{b}^4 \texttt{c} \texttt{a}^3$ \\
		(2) &
	\texttt{b} & 
	\texttt{a} & 
	\texttt{b} & 
	\texttt{a} & 
	\texttt{c} & 
	\texttt{b} & 
	\texttt{a} & 
	\texttt{c} & 
		[ &
		2, &
		7, &
		4$\mid$ &
		1, &
		6, &
		3$\mid$ &
		8, &
		5 &
		] &
		4 & $\texttt{b}^3 \texttt{c}^2 \texttt{a}^3$ \\
		(3) &
		\texttt{a} & 
		\texttt{c} & 
		\texttt{b} & 
		\texttt{a} & 
		\texttt{c} & 
		\texttt{b} & 
		\texttt{a} & 
		\texttt{c} & 
		[ &
		7, &
		4, &
		1$\mid$ &
		6, &
		3$\mid$ &
		8, &
		5, &
		2 &
		] &
		1 & $\texttt{b}^2 \texttt{c}^3 \texttt{a}^3$ \\
		(4) &
		\texttt{a} & 
		\texttt{c} & 
		\texttt{b} & 
		\texttt{a} & 
		\texttt{c} & 
		\texttt{a} & 
		\texttt{c} & 
		\texttt{c} & 
		[ &
		4, &
		1, &
		6$\mid$ &
		3$\mid$ &
		8, &
		5, &
		2, &
		7 &
		] &
		6 & $\texttt{b} \texttt{c}^4 \texttt{a}^3$  \\
		(5) &
		\texttt{a} & 
		\texttt{b} & 
		\texttt{a} & 
		\texttt{b} & 
		\texttt{b} & 
		\texttt{a} & 
		\texttt{b} & 
		\texttt{b} & 
		[ &
		1, &
		6, &
		3$\mid$ &
		8, &
		5, &
		2, &
		7, &
		4 &
		] &
		3 & $\texttt{b}^5 \texttt{a}^3$ \\
		(6) &
		\texttt{c} & 
		\texttt{c} & 
		\texttt{a} & 
		\texttt{c} & 
		\texttt{c} & 
		\texttt{a} & 
		\texttt{c} & 
		\texttt{b} & 
		[ &
		6, &
		3$\mid$ &
		8$\mid$ &
		5, &
		2, &
		7, &
		4, &
		1 &
		] &
		8 & $\texttt{c}^5 \texttt{a}^2 \texttt{b}$ \\
		(7) &
		\texttt{c} & 
		\texttt{c} & 
		\texttt{a} & 
		\texttt{c} & 
		\texttt{b} & 
		\texttt{c} & 
		\texttt{c} & 
		\texttt{b} & 
		[ &
		3$\mid$ &
		8, &
		5$\mid$ &
		2, &
		7, &
		4, &
		1, &
		6 &
		] &
		5 & $\texttt{c}^5 \texttt{a} \texttt{b}^2$ \\
		(8) &
		\texttt{b} & 
		\texttt{a} & 
		\texttt{b} & 
		\texttt{b} & 
		\texttt{a} & 
		\texttt{b} & 
		\texttt{b} & 
		\texttt{a} & 
		[ &
		8, &
		5, &
		2$\mid$ &
		7, &
		4, &
		1, &
		6, &
		3 &
		] &
		2 & $\texttt{b}^5 \texttt{a}^3$ \\
		\bottomrule
	\end{tabular}
	}\caption{\word{} of \cref{eqInducedTernaryWord} for each arithmetically progressed permutation~$P$ of length~$n = 8$ with ratio~$k = 5$, starting with $p_1 := P[1] = k = 5$. The permutation of the $k$-th row is the $k$-th cyclic rotation of the permutation~$P$ in the first row.
	  The splitting of $P$ into the subarrays is visualized by the vertical bar ($\mid$) symbol.
	For (5) and (8), the alphabet is binary and the BWTs are the same.
	The strings of~(3) and~(8) are periodic with period~$n-k$, since the last text position of each subarray is at most as large as $n-k = 3$ ({\it cf.}\ the proof of \cref{thmBinaryStringCharacteristic}).
	For $i \in [1..n]$,  $\BWT_{\word}[i] = \word[P[i+n-k^{-1} \mod n]] = \word[P[i+3 \mod n]]$ with $k^{-1} = k = 5$ defined in \cref{secInversePermutation}.
}
	\label{figTernaryExample}
\end{figure}

Given an arithmetically progressed permutation $P := [p_1, \ldots, p_n]$ with ratio~$k$,
we define the ternary string~$\word[1..n]$ by splitting  $P$ right after the values $n-k$ and $(p_1 - k - 1)\mod n$
into the three subarrays $A$, $B$, and~$C$ (one of which is possibly empty) such that $P = ABC$.
Subsequently, we set
\begin{equation}\label{eqInducedTernaryWord}
\word[p_i] :=
\begin{cases}
	\texttt{a}	& \text{~if~} p_i \in A\text{, or}\\
	\texttt{b}	& \text{~if~} p_i \in B\text{, or}\\
	\texttt{c}	& \text{~if~} p_i \in C\text{.}\\
\end{cases}
\end{equation}
\Cref{figTernaryExample} gives an example of induced ternary/binary strings.

\begin{theorem}\label{thmTernaryAlphabet}
	Given an arithmetically progressed permutation $P := [p_1, \ldots, p_n] \not= [n, \allowbreak n-1, \ldots,1]$ with ratio~$k$,
$\SA_\word = P$ for $\word$ defined in \cref{eqInducedTernaryWord}.
\end{theorem}
\begin{proof}
Suppose we have constructed~$\SA_\word$.
Since $\texttt{a} < \texttt{b} < \texttt{c}$, 
according to the above assignment of~\word{}, the suffixes starting with \texttt{a} lexicographically precede the suffixes starting with \texttt{b}, 
which lexicographically precede the suffixes starting with \texttt{c}.
Hence, $\SA[1..|A|], \SA[|A|+1..|A|+|B|]$ and $\SA[|A|+|B|+1..n]$ store the same text positions as $A$, $B$, and~$C$, respectively.
Consequently, it remains to show that the entries of each subarray ($A$, $B$ or $C$) are also sorted appropriately.
Let $p_i$ and $p_{i+1}$ be two neighboring entries within the same subarray.
Thus, $\word[p_i] = \word[p_{i+1}]$ holds, 
and the lexicographic order of their corresponding suffixes~$\word[p_i..n]$ and $\word[p_{i+1}..n]$ is determined by comparing the subsequent positions, 
starting with $\word[p_i + 1]$ and $\word[p_{i+1} + 1]$.
Since we have $(p_{i+1} + 1) - (p_i + 1) = p_{i+1} - p_i = k$, 
we can recursively show that these entries remain in the same order next to each other in the suffix array until either reaching the last array entry or a subarray split, that is, (1) $p_i + 1 = p_1 - k$ or (2) $p_i = n-k$. 

\begin{enumerate}
\item
When $p_i + 1$ becomes $p_1 - k \mod n = p_n$ (the last entry in \SA),
$p_{i+1}$ is in the subsequent subarray of the subarray of
$p_i = p_1 - k -1$ (remember that $A$ or $B$ ends directly after $p_1 - k -1$, {\it cf.}\ \cref{figProgressToLastEntry}).
Hence $\word[p_i] < \word[p_{i+1}]$, and $\word[p_i .. n] \prec \word[p_{i+1} .. n]$.

\item
The split at the value~$n-k$ ensures that when reaching $p_i = n-k$ and $p_{i+1} = n$, we can stop the comparison here as there is no character following $\word[p_{i+1}]$.
The split here ensures that we can compare the suffixes $\word[n-k..n]$ and $\word[n]$ by the characters $\word[n-k] < \word[n]$.
	If we did not split here, $\word[n] = \word[n-k]$, and the suffix $\word[n]$ would be a prefix of $\word[n-k..n]$, resulting in $\word[n] \prec \word[n-k..n]$ (which yields a contradiction unless $p_n = n$).
\end{enumerate}

To sum up, the text positions stored in each of $A$, $B$ and~$C$ are in the same order as in $\SA_{\word}$ since 
the $j-1$ subsequent text positions of each consecutive pair of entries $p_i$ and $p_{i+1}$ are consecutive in~$P$ for
the smallest integer $j \in [1..n]$ such that $p_{i+1} + jk \in \lbrace p_1 - 1, n \rbrace$.

\end{proof}

\begin{figure}[tb]
\centering{\includegraphics[width=\linewidth]{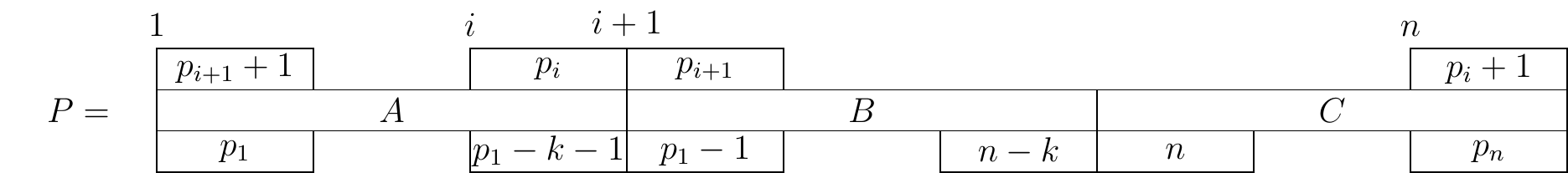}
  }\caption{Setting of the proof of \cref{thmTernaryAlphabet} with the condition~$p_i + 1 = p_1 - k = p_n$. In the figure we assume that the entry $p_1 - k -1$ appears before $n-k$ in~$P$.}
  \label{figProgressToLastEntry}
\end{figure}

\begin{figure}
	\centering{\hfill
\begin{minipage}{0.2\textwidth}
BWT matrix of 
\texttt{babbabac}:\\
\\
\texttt{abacbab\textsl{b}}\\
\texttt{abbabac\textsl{b}}\\
\texttt{acbabba\textsl{b}}\\
\texttt{babacba\textsl{b}}\\
\texttt{babbaba\textsl{c}}\\
\texttt{bacbabb\textsl{a}}\\
\texttt{bbabacb\textsl{a}}\\
\texttt{cbabbab\textsl{a}}\\
\end{minipage}
\hfill
\begin{minipage}{0.2\textwidth}
BWT matrix of 
\texttt{ccaccacb}:\\
\\
\texttt{acbccac\textsl{c}}\\
\texttt{accacbc\textsl{c}}\\
\texttt{bccacca\textsl{c}}\\
\texttt{cacbcca\textsl{c}}\\
\texttt{caccacb\textsl{c}}\\
\texttt{cbccacc\textsl{a}}\\
\texttt{ccacbcc\textsl{a}}\\
\texttt{ccaccac\textsl{b}}\\
\end{minipage}
\hfill
\begin{minipage}{0.2\textwidth}
BWT matrix of 
\texttt{bbabbabb}:\\
\\
\texttt{abbabbb\textsl{b}}\\
\texttt{abbbbab\textsl{b}}\\
\texttt{babbabb\textsl{b}}\\
\texttt{babbbba\textsl{b}}\\
\texttt{bbabbab\textsl{b}}\\
\texttt{bbabbbb\textsl{a}}\\
\texttt{bbbabba\textsl{b}}\\
\texttt{bbbbabb\textsl{a}}\\
\end{minipage}
\hfill
	}\caption{BWTs defined by the lexicographic sorting of all rotations of strings whose suffix arrays are cyclic rotations.
		This figure shows (from left to right) the BWT matrices of the strings of Rotation~(1) and~(6) of \cref{figTernaryExample} as well as of Case~(2) from \cref{figBinaryExample}.
		Reading the last column of a BWT matrix (whose characters are italic) from top downwards yields the BWT defined on the BWT matrix.
	While the BWT defined on the BWT matrix and the one defined by the suffix array coincides for the strings of \cref{eqInducedTernaryWord} due to \cref{thmMatrixBWT}, this is not the case in general for the binary strings studied in \cref{secBinary}, where we observe that
	$\BWT_{\texttt{bbabbabb}} = \texttt{bbbbbaab}$ defined by the suffix array differs from \texttt{bbbbbaba} (the last column on the right).
}
	\label{figBWTmatrix}
\end{figure}

Knowing the suffix array of the ternary string~$\word$ of \cref{eqInducedTernaryWord}, we can give a characterization of its BWT\@.
We start with the observation that both BWT definitions (rotation based and suffix array based) coincide for the strings of \cref{eqInducedTernaryWord} (but do not in general as highlighted in the introduction, \textit{cf.}~\cref{figBWTmatrix}), and then continue with insights in how the BWT looks like.

\begin{theorem}\label{thmMatrixBWT}
	Given an arithmetically progressed permutation $P := [p_1, \ldots, p_n]  \not= [n, \allowbreak n-1, \ldots,1] $ with ratio~$k$ 
	and the string~$\word$ of \cref{eqInducedTernaryWord},
	the BWT of \word{} defined on the BWT matrix coincides with the BWT of \word{} defined on the suffix array.
\end{theorem}
\begin{proof}
	According to \cref{thmTernaryAlphabet}, $\SA_{\word} = P$, and therefore the BWT of \word{} defined on the suffix array is given by $\BWT_\word[i] = \word[p_i - 1 \bmod n]$.
    The BWT matrix is constituted of the lexicographically ordered cyclic rotations of $\word$.
	The BWT $\BWT_{\textup{matrix}}$ based on the BWT matrix is obtained by reading the last column of the BWT matrix from top downwards (see \cref{figBWTmatrix}).
	Formally, let $Q[i]$ be the starting position of the lexicographically $i$-th smallest rotation $\word[Q[i]..n]\word[1..Q[i]-1]$.
	Then $\BWT_{\textup{matrix}}[i]$ is $\word[Q[i]-1 \mod n]$ if $Q[i] > 1$, 
	or $\word[n]$ if $Q[i] = 1$.
    We prove the equality $P = Q$ by showing that, for all $i \in [1..n-1]$, the rotation $R_{i} := \word[p_i..n]\word[1..p_i-1]$ starting at $p_i = \SA_\word[i]$ is lexicographically smaller than the rotation 
	$R_{i+1} :=\word[p_{i+1}..n]\word[1..p_{i+1}-1]$
	starting at $p_{i+1} = \SA_\word[i+1]$.
	We do that by comparing both rotations $R_{i}$ and $R_{i+1}$ characterwise:
    
Let $j$ be the first position where $R_i$ and $R_{i+1}$ differ, i.e., 
	$R_i[j] \neq R_{i+1}[j]$ and $R_i[q] = R_{i+1}[q]$ for every $q \in [1..j)$.

First we show that $j \not= p_n - p_i + 1 \bmod n$ by a contradiction:
	Assuming that $j = p_n - p_i + 1 \bmod n$,
	we conclude that $j \neq 1$ by the definition of $i \in [1..n-1]$.
	Since $k$ is the ratio of~$P$,
        we have 
		\[
		R_i[j - 1] = \word[p_i + j - 2 \bmod n] =  \word[p_n - 1 \bmod n] = \word[p_1 - k - 1 \bmod n]
	\]
		and 
		$R_{i+1}[j-1] = \word[p_1 - 1 \bmod n]$.
        By \cref{eqInducedTernaryWord},
		$p_1 - k - 1 \mod n$ and $p_1 - 1 \mod n$ belong to different subarrays of~$P$, 
		therefore 
		$\word[p_1 - k - 1] \neq \word[p_1 - 1]$ and 
		$R_i[j - 1] \neq R_{i+1}[j -1]$,
		contradicting the choice of $j$ as the first position where $R_i$ and $R_{i+1}$ differ.

		This concludes that $j \le n$ (hence, $R_i \not= R_{i+1}$) and $j \neq p_n - p_i + 1 \bmod n$.
			Hence, $R_i[j] = \word[p_i + j - 1 \bmod n]$ and $R_{i+1}[j] = \word[p_{i+1} + j - 1 \bmod n] = \word[p_i + j - 1 + k \bmod n]$
			are characters given by two consecutive entries in $\SA_{\word}$, i.e.,
			$\SA_{\word}[q] = p_i + j - 1 \bmod n$ and $\SA_{\word}[q+1] = p_i + j - 1 + k \bmod n$ for a $q \in [1..n-1]$.
			Thus $R_i[j] \le R_{i+1}[j]$, and by definition of $j$ we have $R_i[j] < R_{i+1}[j]$, leading finally to $R_i \prec R_{i+1}$.
			Hence, $Q = P$.
\end{proof}

\begin{figure}[tb]
\centering{\includegraphics[width=0.47\linewidth]{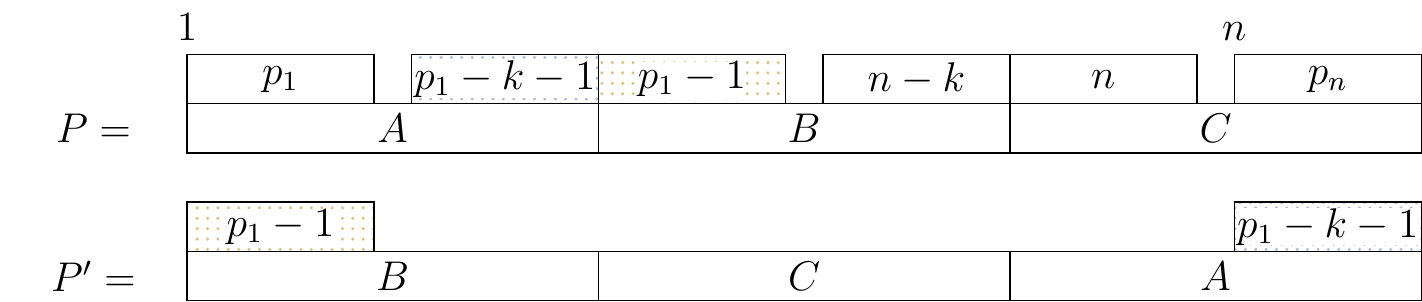}
    \includegraphics[width=0.47\linewidth]{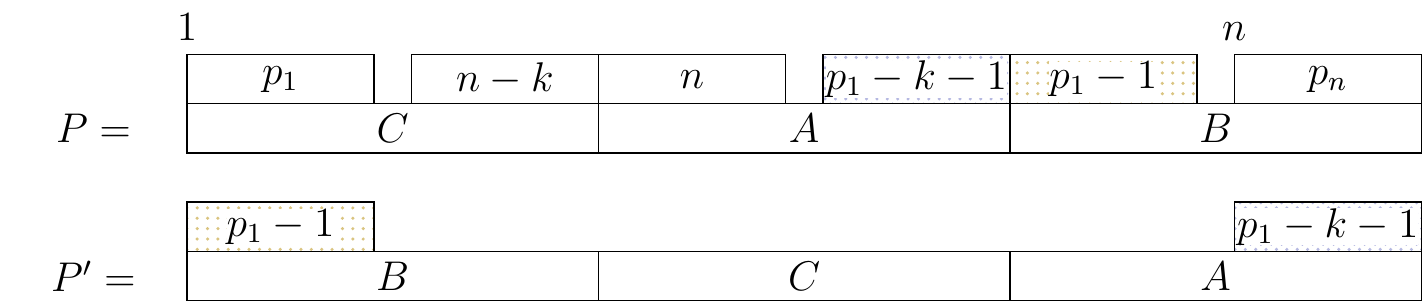}
  }\caption{Setting of \cref{eqInducedTernaryWord}
	  with the distinction whether the entry $p_1 - k -1$ appears before (left) or after (right) $n-k$ in~$P$, yielding a different shape of the $\BWT_{\word}$ defined as $\BWT_{\word[i]} = \word{}[P'[i]]$ with $P'[i] = \SA_{\word}[i]-1 \mod n$.
  }
	  \label{figBwtSubarray}
\end{figure}

\begin{lemma}\label{lemTernaryBWT}
	Let $P := [p_1, \ldots, p_n]  \not= [n, \allowbreak n-1, \ldots,1]$ be an arithmetically progressed permutation with ratio~$k$.
Further, let~$\word[1..n]$ be given by \cref{eqInducedTernaryWord} such that $\SA_\word = P$ according to \cref{thmTernaryAlphabet}.
Given that $p_t = p_1 - 1 - k \mod n$ for a $t \in [1..n]$,
$\BWT_{\word}$ is given by the $t$-th rotation of $\word[\SA[1]] \cdots \word[\SA[n]]$, i.e.,
$\BWT_{\word}[i] = \word[P[i+t \mod n]]$ for $i \in [1..n]$.
\end{lemma}
\begin{proof}
Since $P$ is an arithmetically progressed permutation with ratio~$k$ 
then so is the sequence $P' := [p'_1,\ldots,p'_n]$ with $p'_i = p_i - 1 \mod n$.
In particular,  $P'$ is a cyclic shift of $P$ with
$p'_n = p_1 - 1 - k \mod n$ because $p'_1 = p_1-1$.
However, $p_1 - 1 - k$ is a split position of one of the subarrays~$A$, $B$, or $C$,
meaning that $P'$ starts with one of these subarrays and ends with another of them ({\it cf.}~\cref{figBwtSubarray}).
Consequently, there is a $t$ such that $p_t = p'_n$, and 
we have the property that $\BWT_{\word}$ with $\BWT_{\word}[i] = \word[P'[i]]$ is the $t$-th rotation of $\word[\SA[1]] \cdots \word[\SA[n]]$.
\end{proof}
We will determine the parameter~$t = n - k^{-1} \mod n$ after \cref{eqInducedBinaryInverse} in \cref{secInversePermutation}, 
where $k^{-1}$ is defined such that $k \cdot k^{-1} \mod n = 1 \mod n$.
With Lemma~\ref{lemTernaryBWT},  we obtain the following corollary which shows that the number of runs in $\BWT_\word{}$ for \word{} defined in \cref{eqInducedTernaryWord} are minimal:

\begin{corollary}
	\label{corSimpleBWT}
	For an arithmetically progressed permutation  $P := [p_1, \ldots, p_n]  \not= [n, \allowbreak n-1, \ldots,1]$ and the string~$\word$ defined by \cref{eqInducedTernaryWord},
$\BWT_{\word}$ consists of exactly 2 runs if $\word$ is binary, while it consists of exactly 3 runs if $\word$ is ternary.
\end{corollary}

\begin{theorem}
\label{thmAPPunique}
Given an arithmetically progressed permutation $P := [p_1, \ldots, p_n]$ with ratio~$k$ such that  $p_1 \not\in \lbrace 1, k+1, n \rbrace$,
the string \word{} given in \cref{eqInducedTernaryWord} is unique.
\end{theorem}
\begin{proof}
	The only possible way to define another string~$\word'$ would be to change the borders of the subarrays~$A$, $B$, and~$C$.
	Since $p_1 \not\in \lbrace 1, n \rbrace $, $n-k$ and $n$, as well as $p_1 - k - 1$ and $p_1 - 1$, are stored as a consecutive pair of text positions in~$P$.
	\begin{itemize}
		\item 
			If $P$ is not split between its consecutive text positions~$n-k$ and~$n$, 
			then $\word'[n-k] = \word'[n]$.
			Consequently, we have the contradiction $\word'[n] \prec \word'[n-k..n]$.
\item 
	If $P$ is not split between its consecutive text positions~$(p_1 - k - 1) \mod n$ and~$(p_1 - 1) \mod n$, 
	then $\word'[p_1 - k - 1 \mod n] = \word'[p_1 - 1 \mod n]$.
	Since $p_1 \neq k+1$, and $\word'[p_1 - k \mod n] = \word'[p_n] > \word'[p_1]$, 
	this leads to the contradiction $\word'[(p_1 - k - 1 \bmod n)..n] \succ \word'[(p_1 - 1 \bmod n)..n]$,
{\it cf.}\ \cref{figProgressToLastEntry}.
	\end{itemize}
\end{proof}

Following this analysis of the ternary case we proceed to consider binary strings.
A preliminary observation is given in \cref{figTernaryExample},
which shows, for the cases $p_1$ is 1 and $n$ in Theorem~\ref{thmAPPunique}, namely Rotations (5) and (8), that a rotation of $n-k$ in the permutation gives a rotation of one in the corresponding binary strings.  
We formalize this observation in the following lemma, drawing a connection between binary strings whose suffix arrays are arithmetically progressed and start with $1$ or $n$.

\begin{lemma}
\label{lemRotate}
Let $P := [p_1, \ldots, p_n]$ be an arithmetically progressed permutation with ratio $k$ and $p_1 = 1$ for a binary string~\word{} over $\Sigma = \{\mathtt{a},  \mathtt{b}\}$ with $\SA_{\word{}} = P$.
Suppose that the number of $\mathtt{a}$'s in~\word{} is $m$ and that $\word' = \word[2] \cdots \word[n]\word[1]$ is the first rotation of~\word.
Then $\SA_{\word'}$ is the $m$-th rotation of $P$ with $\SA_{\word'}[1] = n$.
Furthermore, $\BWT_{\word} = \BWT_{{\word}'}$.
\end{lemma}

\begin{proof}
Since $p_1 = 1$, $\word[1] = \texttt{a}$ and $\word[n] = \texttt{b}$.
In the following, we show that $P' = \SA_{\word'}$ for $P' := [p'_1, \ldots, p'_n] := [p_1-1 \bmod n, \ldots , p_{n}-1 \bmod n]$ with $p'_1 = p_1-1 = n$ (since $\word'[n] = \texttt{a}$).
For that, we show that each pair of suffixes in $\SA_{\word}$ is kept in the same relative order in $P'$ (excluding $\SA_{\word}[1] = 1$):

Consider two text positions $p_i, p_j \in [p_2, \ldots, p_{n}]$ with
$\word[p_i..n] 
= u_1 \cdots u_s 
\prec \word[p_j..n] 
= v_1 \cdots v_t$. 
\begin{itemize}
	\item If $u_h \neq v_h$ for the least $h \in [1..\min\{s,t\}]$, 
		then $u_1 \cdots u_s \texttt{a} \prec v_1 \cdots v_t \texttt{a}$. 
	\item Otherwise, $u_1 \cdots u_s$ is a proper prefix of $v_1 \cdots v_t = u_1 \cdots u_s v_{s+1} \cdots v_t$. 
		\begin{itemize}
			\item If $v_{s+1} = \texttt{b}$, then
			$u_1 \cdots u_s \texttt{a} \prec u_1 \cdots u_s  \texttt{b}  v_{s+2} \cdots v_t \texttt{a} = v_1 \cdots v_t \texttt{a}$.
			\item Otherwise ($v_{s+1} = \texttt{a}$), 
			$u_1 \cdots u_s \texttt{a}$ is a proper prefix of $v_1 \cdots v_t \texttt{a}$,
			and similarly 
			$u_1 \cdots u_s \texttt{a} \prec u_1 \cdots u_s  \texttt{a}  v_{s+2} \cdots v_t \texttt{a} = v_1 \cdots v_t \texttt{a}$.  
		\end{itemize}
\end{itemize}
Hence the relative order of these suffixes given by $[p_2,\ldots,p_{n}]$ and $[p'_2, \ldots,p'_{n}]$ is the same. 
In total, we have $p'_i = p_i - 1 \mod n$ for $i \in [1..n]$, hence $P'$ is an arithmetically progressed permutation with ratio $k$.
Given the first $m$ entries in $P$ represent all suffixes of~\word{} starting with \texttt{a},
$P'$ is the $m$-th rotation of $P$ since
$p'_1 = n$ is the $(m+1)$-th entry of $P$, i.e., the smallest suffix starting with \texttt{b} in~\word{}.
Finally, since the strings $\word{}$ and $\word'$ are rotations of each other, their BWTs are the same.
\end{proof}
Like the parameter~$t$ of Lemma~\ref{lemTernaryBWT}, we will determine the parameter~$m$ after \cref{eqInducedBinaryInverse} in \cref{secInversePermutation}.

\subsection{Binary Alphabet}\label{secBinary}

We start with the construction of a binary string from an arithmetically progressed permutation:

\begin{theorem}\label{thmBinaryAlphabet}
	Given an arithmetically progressed permutation $P := [p_1, \ldots, p_n] \not= [n, \allowbreak n-1,\ldots,1]$ with ratio~$k$ such that  $p_1 \in \lbrace 1, k+1, n \rbrace$,
we can modify~\word{} of \cref{eqInducedTernaryWord} to be
a string over the binary alphabet $\{\mathtt{a},\mathtt{b}\}$ with
$\SA_\word = P$.
\end{theorem}
\begin{proof}
If $p_1 = 1$, then $P$ is split after the occurrences of the values~$n-k$ and $-k = n - k \mod n$, which gives only two non-empty subarrays.
If $p_1 = n$, $P$ is split after the occurrence of~$n-k-1$, which implies that $C$ is empty since $p_n = n - k$.
Hence, \word{} can be constructed with a binary alphabet in those cases, {\it cf.} Fig.~\ref{figTernaryExample}.

For the case $p_1 = k+1$, $P$ is split after the occurrences of the values~$n-k$ and $k+1-k-1 \mod n = n \mod n$, 
so $B$ contains only the text position~$n$.
By construction, the requirement is that the suffix $\word[n]$ is smaller than all other suffixes starting with \texttt{c}.
So instead of assigning the unique character $\word[n] \gets \texttt{b}$ like in \cref{thmTernaryAlphabet},
we can 
assign $\word[n] \gets \texttt{c}$, which still makes $\word[n]$ the smallest suffix starting with \texttt{c}.
We conclude this case by converting the binary alphabet $\{\texttt{a},\texttt{c}\}$ to $\{\texttt{a},\texttt{b}\}$. {\it Cf.} Fig.~\ref{figTernaryExample}, where \word{} in Rotation (6) has become $\texttt{bbabbabb}$ with period $n-k=3$.
\end{proof}

The main result of this section is the following theorem.
There, we characterize all binary strings whose suffix arrays are arithmetically progressed permutations.
More precisely, we identify which of them are unique\footnote{The exact number of these binary strings is not covered by Theorem~\ref{thmAPPunique}.},
periodic, or a Lyndon word.

\newcommand*{\wordS}{\ensuremath{\word_s}}
\newcommand*{\wordC}{\ensuremath{\word_{\hspace{-0.1em}\textup{c}}}}

\begin{figure}
  \setlength{\tabcolsep}{0.4em}
	\centering
	\begin{tabular}{l
@{\hskip 1em}
*{8}{l}
@{\hskip 1em}
r*{8}{c}l
@{\hskip 1em}
c@{\hskip 1em}c}
		\toprule
		Case & \multicolumn{8}{c}{\word{}} & \multicolumn{10}{c}{\SA{}} & $p_s$ & $s$ \\
		\cmidrule(r){2-9}
		\cmidrule{11-18}
		&
		{\scriptsize 1} &
		{\scriptsize 2} &
		{\scriptsize 3} &
		{\scriptsize 4} &
		{\scriptsize 5} &
		{\scriptsize 6} &
		{\scriptsize 7} &
		{\scriptsize 8} &
		  &
		{\scriptsize 1} &
		{\scriptsize 2} &
		{\scriptsize 3} &
		{\scriptsize 4} &
		{\scriptsize 5} &
		{\scriptsize 6} &
		{\scriptsize 7} &
		{\scriptsize 8} &
		  &
		\\
		(1) & 
		\texttt{b} & 
		\texttt{a} & 
		\texttt{b} & 
		\texttt{b} & 
		\texttt{a} & 
		\texttt{b} & 
		\texttt{b} & 
		\texttt{a} & 
	[ &
		8, &
		5, &
		2, &
		7, &
		4, &
		1, &
		6, &
		3 &
		] &
		2 & 3 \\
		(2) &
	\texttt{b} & 
	\texttt{b} & 
	\texttt{a} & 
	\texttt{b} & 
	\texttt{b} & 
	\texttt{a} & 
	\texttt{b} & 
	\texttt{b} & 
		[ &
		6, &
		3, &
		8, &
		5, &
		2, &
		7, &
		4, &
		1 &
		] &
		3 & 2 \\
		(3) &
		\texttt{a} & 
		\texttt{b} & 
		\texttt{a} & 
		\texttt{b} & 
		\texttt{b} & 
		\texttt{a} & 
		\texttt{b} & 
		\texttt{b} & 
		[ &
		1, &
		6, &
		3, &
		8, &
		5, &
		2, &
		7, &
		4 &
		] &
		3 & 3 \\
		\bottomrule
	\end{tabular}
	\caption{All binary strings of length~$8$ whose suffix arrays are arithmetically progressed permutations with ratio~$k=5$. 
		\cref{thmBinaryStringCharacteristic} characterizes these strings (and also gives the definition of~$p_s$).
		Cases~(1) and~(3) also appear in \cref{figTernaryExample} at Rotation~(8) and~(5), respectively, while
		Case~(2) can be obtained from Rotation~(6) by exchanging the last character with~\texttt{c}. Cases~(1) and~(2) both have period $n-k=3$, and Case~(3) is a Lyndon word. 
	}
	\label{figBinaryExample}
\end{figure}

\begin{figure}
	\centering{\begin{tabular}{ccc}
		\includegraphics[width=0.3\linewidth]{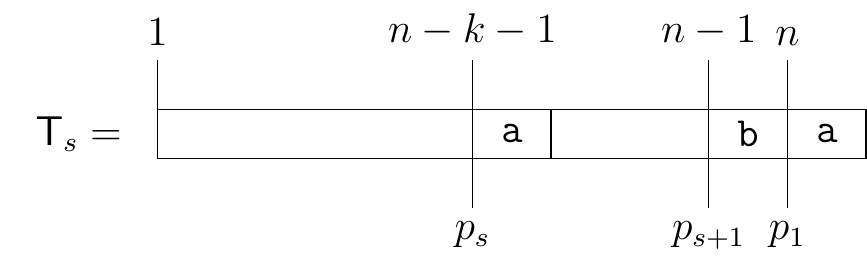}
		&
		\includegraphics[width=0.3\linewidth]{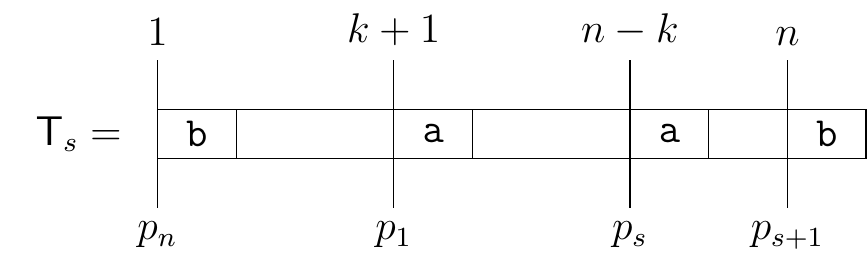}
		&
		\includegraphics[width=0.3\linewidth]{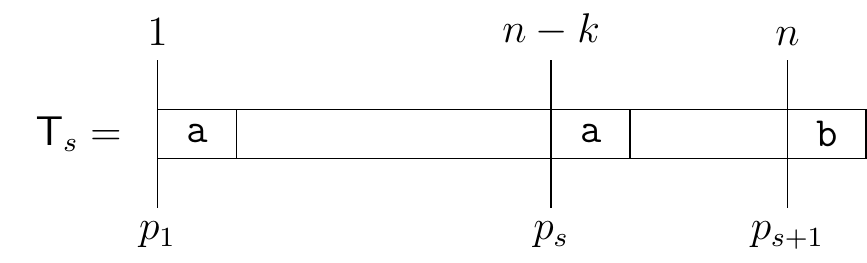}
		\\
			Case~1 & Case~2 & Case~3
		\end{tabular}
	}\caption{Sketches of the cases of \cref{thmBinaryStringCharacteristic}. 
	$\wordS$ is uniquely determined if the suffix array~\SA{} of $\wordS$ is arithmetically progressed with ratio~$k$ and the first entry~$\SA[1] \in \{1, k+1, n-k\}$ is given. }
	\label{figBinaryStringCharacteristic}
\end{figure}

\newcommand*{\String}{\ensuremath{\mathsf{S}}}

\begin{theorem} \label{thmBinaryStringCharacteristic}
	Let $n$ and $k \in [1..n-1]$ be two coprime integers.
	If $k \not= n-1$,
there are \emph{exactly three} binary strings of length $n$ whose suffix arrays are arithmetically progressed permutations with ratio $k$.
Each such solution $\wordS \in \{\mathtt{a},\mathtt{b}\}^+$ 
is characterized by 
\begin{equation}\label{eqInducedBinaryWord}
\wordS[i] = 
\begin{cases}
\mathtt{a} &  \text{~for~} i \in \SA_{\wordS}[1..s] \text{, or} \\
\mathtt{b} &  \text{~otherwise,}
\end{cases}
\end{equation}
for all text positions~$i \in [1..n]$ and an index~$s \in [1..n-1]$, called the \Define{split index}.

	The individual solutions are obtained by fixing the values for $p_1$ and 
	$p_s$, the position of the lexicographically largest suffix starting with $\mathtt{a}$, of $\SA_{\wordS} = [p_1, \ldots, p_n]$:
	\begin{enumerate}
		\item $p_1 = n$ and $p_s = n-k-1$, 
		\item $p_1 = k+1$ and $p_s = n-k$, and
		\item $p_1 = 1$ and $p_s = n-k$.
	\end{enumerate}
	The string~\wordS{} has 
period~$n-k$ in Cases~1 and~2,
	while \wordS{} of Case~3 is a Lyndon word, which is not periodic by definition.
	
	For $k=n-1$, Cases~2 and~3 each yields exactly one binary string, but Case~1 yields $n$ binary strings according to \cref{thmUnaryEnumeration}.
\end{theorem}
\begin{proof}
	Let \String{} be a binary string of length~$n$, and suppose that $\SA_{\String} = P := [p_1, \ldots, p_n]$ is an arithmetically progressed permutation with ratio~$k$.
	Further let $p_s$ be the position of the largest suffix of $\String{}$ starting with \texttt{a}.
	Then $\String[p_i..n] \prec \String[p_{i+1}..n]$ and thus $\String[p_i] \le \String[p_{i+1}]$.
	We have $\String[j] = \String[j+k \mod n]$ for all $j \in [1..n] \setminus \{p_n, p_s\}$ since 
	\begin{itemize}
		\item $p_n = \SA_{\String}[n]$ is the starting position of the largest suffix 
			($\String[p_n] = \texttt{b} \not= \texttt{a} = \String[p_1] = \String[p_n + k]$).
		\item $\String[p_s..n]$ and $\String[p_{s+1}..n]$ are the lexicographically largest suffix starting with \texttt{a} and the lexicographically smallest suffix starting with \texttt{b}, respectively, such that $\String[p_{s}] = \texttt{a} \not= \texttt{b} = \String[p_{s+1}]$.
	\end{itemize}
	To sum up, since $\String[p_s..n] \prec \String[p_{s+1}..n]$ by construction, $\String[p_i..n] \prec \String[p_{i+1}..n]$ holds for $p_i > p_{i+1}$ whenever $p_{i} \not= p_n$.
	This, together with the coprimality of $n$ and $k$, determines $p_s$ uniquely in the three cases
	({\it cf.} \cref{figBinaryExample} for the case that $n=8$ and $k=5$, and \cref{figBinaryStringCharacteristic} for sketches of the proof):
	\begin{enumerate}
		\item[Case~1:] 
		We first observe that the case $k = n-1$ gives us $P = [n,n-1,\ldots,1]$, and this case was already treated with \cref{thmUnaryAlphabet}.
		In the following, we assume $k < n-1$, and under this assumption we have
			$s > 1$, $\wordS[n] = \wordS[p_1] =\texttt{a}$ and $\wordS[n-1] = \texttt{b}$ 
			(otherwise $\wordS[n-1..n]$ would be the second smallest suffix, i.e., $P[2] = n-1$ and hence $k = n-1$).
			Consequently, $\wordS[n-1..n] = \wordS[p_{s+1}..n]$ is the smallest suffix starting with \texttt{b}, namely \texttt{ba}, and therefore $p_s = n-1-k$.

		\item[Case~$p_1 \not= n$:]
			If $p_1 \not= n$, then $\wordS[n] = \texttt{b}$ (otherwise $\wordS[n] \prec \wordS[p_1..n]$).
			Therefore, $\wordS[n]$ is the smallest suffix starting with \texttt{b}, and consequently $p_s = n-k$.
\end{enumerate}
	For the periodicity, with
	$\wordS[j] = \wordS[j + k \mod n] = \wordS[j - (n-k) \mod n]$ for $j \in [1..n] \setminus \{p_1, p_s\}$
	we need to check two conditions:
	\begin{itemize}
		\item If $p_n - (n - k) > 0$, then $\wordS[p_n - (n - k)] = \wordS[p_1] \not= \wordS[p_n]$ breaks the periodicity.
		\item If $p_s - (n - k) > 0$, then $\wordS[p_s - (n-k)] = \texttt{a} \not= \texttt{b} = \wordS[p_{s+1}]$ breaks the periodicity.
	\end{itemize}
For Case~1, $p_n = n-k$ and $p_s = n-k-1$ (hence $p_n - (n - k) =0$ and $p_s - (n-k) = -1$), thus Case~1 is periodic.

\noindent Case~2 is analogous to Case~1.

\noindent For Case~3, \wordS{} does not have period~$n-k$ as $p_n = n-k+1$, and hence $p_n - (n - k) > 0$.
It cannot have any other period since Case~3 yields a Lyndon word (because the lexicographically smallest suffix~$\wordS[p_1..n] = \wordS[1..n]$ starts at the first text position).
Note that Case~3 can be obtained from Case~2 by setting $\wordS[1] \gets \texttt{a}$ (the smallest suffix $\wordS[k+1..n]$ thus becomes the second smallest suffix).

Finally, we 
need to show that no other value for $p_1$ admits a binary string~\String{} having an arithmetically progressed permutation~$P := [p_1,\ldots,p_n]$ with ratio~$k$ as its suffix array.
So suppose that $p_1 \notin \{1, k+1, n\}$, then this would imply the following:
	\begin{itemize}
		\item $\String[p_1] = \texttt{a}$ because the smallest suffix starts at text position~$p_1$, and
		\item $\String[p_1-1] = \texttt{b}$ because of the following:
		First, the text position $\String[p_1-1]$ exists due to $p_1 > 1$.
		Second, since $p_1 < n$, there is a text position~$j \in [p_1+1..n]$ such that $\String[p_1] = \ldots = \String[j-1] = \texttt{a}$ and $\String[j] = \texttt{b}$ (otherwise $\String[n]$ would be the smallest suffix).
		If $\String[p_1-1] = \texttt{a}$, then the suffix~$\String[p_1-1..n]$ starting with~$\texttt{a}^{j-p_1+1}\texttt{b}$  
		is lexicographically smaller than the suffix~$\String[p_1..n]$ starting with~$\texttt{a}^{j-p_1}\texttt{b}$.
		Hence, $\String[p_1-1] = \texttt{b}$ must hold.
\item $p_n-1 \ge 1$ (since $p_1 \not= k+1$) and $\String[p_n-1] = \texttt{a}$.
		If $\String[p_n-1] = \texttt{b}$, then the suffix $\String[p_n-1..n]$ has a longer prefix of \texttt{b}'s than the suffix $\String[p_n..n]$, and is therefore lexicographically larger.
	\end{itemize}
	
Since $\String[p_n-1] = \texttt{a}$ and $\String[p_1-1] = \texttt{b}$ with $p_n-1+k \mod n = p_1-1$,
	the smallest suffix starting with \texttt{b} is located at index $p_1-1$.
	This is a contradiction as $p_1 \neq n$ implies $\String[n] = \texttt{b}$ 
	(if $\String[n] = \texttt{a}$, then $\SA[1] = n$ instead of $\SA[1] = p_1$) and thus the smallest suffix starting with \texttt{b} is located at index $n$ (this is a contradiction since we assumed that this suffix starts at $p_1 - 1 \in [1..n-1]$).
	This establishes the claim for $p_1$.
\end{proof}

For a given arithmetically progressed permutation with ratio~$k$, and first entry $p_1 \in \{1,k+1,n\}$, the string~\wordS{} of \cref{thmBinaryStringCharacteristic}
coincides with $\word$ of \cref{thmBinaryAlphabet}.\\

\subsection{Inverse Permutations}\label{secInversePermutation}
Since the inverse~$P^{-1}$ of a permutation~$P$ with $P^{-1}[P[i]] = i$ is also a permutation, 
one may wonder whether the inverse~$P^{-1}$ of an arithmetically progressed permutation is also arithmetically progressed.
We affirm this question in the following.
For that, we use the notion of the \Define{multiplicative inverse} $k^{-1}$ of an integer~$k$ (to the congruence class $[1..n] = \mathbb{Z}/n\mathbb{Z}$),
which is given by $k^{-1} \cdot k \mod n = 1 \mod n$.
The multiplicative inverse~$k^{-1}$ is uniquely defined if $k$ and $n$ are coprime.

\begin{theorem}\label{thmInversePermutation}
	The inverse~$P^{-1}$ of an arithmetically progressed permutation~$P$ with ratio~$k$ is an arithmetically progressed permutation with ratio~$k^{-1}$ and 
	$P^{-1}[1] = (1 - P[n]) \cdot k^{-1} \bmod n$.
\end{theorem}
\begin{proof}
	Let $x := P[i]$ for an index~$i \in [1..n]$.
	Then $P[i+k^{-1} \bmod n] = P[i] + k\cdot k^{-1} = x + k\cdot k^{-1} \bmod n = x+1\bmod n$.
	For the inverse permutation~$P^{-1}$ this means that $P^{-1}[x] = i$ and $P^{-1}[x+1 \bmod n] = i+k^{-1} \bmod n$.
Thus the difference $P^{-1}[x+1\bmod n] - P^{-1}[x]$ is $k^{-1}$.

Since $P[i] = j \Longleftrightarrow P[n] + ik \mod n = j$ holds for all indices~$i \in [1..n]$, we have
(using $i \gets P^{-1}[1]$ and $j \gets 1$ in the above equivalence)
\begin{align*}
	P[P^{-1}[1]] = 1 \mod n &\Longleftrightarrow P[n] + P^{-1}[1]\cdot k = 1 \mod n \\
							&\Longleftrightarrow P^{-1}[1] \cdot k = 1 - P[n] \mod n \\
							&\Longleftrightarrow P^{-1}[1] = (1 - P[n]) \cdot k^{-1} \mod n.
\end{align*}
\end{proof}

Consequently, using the split index $s$ of $p_s$ for \SA{} and 
\begin{align*}
	\ISA[i] 
	&= \ISA[1] + (i-1) k^{-1}  \mod n \\
	&= (1 - \SA[n]) \cdot k^{-1} + (i-1) k^{-1}  \mod n \\
	&= (i - \SA[n]) \cdot k^{-1} \mod n,
\end{align*}
we can rewrite $\wordS$ defined in \cref{eqInducedBinaryWord} as

\begin{equation}\label{eqInducedBinaryInverse}
\wordS[i] =
\begin{cases}
	\texttt{a}	& \text{~if~} \ISA[i] \le ~s \text{, or} \\
	\texttt{b}	& \text{~otherwise} 
\end{cases}
\end{equation}
where \SA{} and \ISA{} denote the suffix array and the inverse suffix array of~$\wordS$, respectively. 
Another result is that 
$\ISA[p_s] = s$ is the number of \texttt{a}'s in \wordS{}, for which we split the study into the cases of \cref{thmBinaryStringCharacteristic}:
	\begin{enumerate}
		\item If $\SA[1] = n$ and $p_s = n-k-1$, then $\SA[n] = n-k$ and $\ISA[i] = (i - n + k) k^{-1} \mod n$.
Consequently, $s = \ISA[p_s] = (-1) k^{-1} \mod n = n - k^{-1} \mod n$.
		\item If $\SA[1] = k+1$ and $p_s = n-k$, then $\SA[n] = 1$ and $\ISA[i] = (i - 1)  k^{-1} \mod n$.
Consequently, $s = \ISA[p_s] = (n-k-1) k^{-1} \mod n = n k^{-1} -1 - k^{-1} \mod n = n - 1 - k^{-1} \mod n$.
		\item If $\SA[1] = 1$ and $p_s = n-k$, then $\SA[n] = n-k+1$ and $\ISA[i] = (i - n + k-1) k^{-1} \mod n$.
Consequently, $s = \ISA[p_s] = (-1) k^{-1} \mod n = n - k^{-1} \mod n$ as in Case~(1).
	\end{enumerate}
For \cref{figBinaryExample} with $k = 5$ and $n = 8$, 
we know that the number of \texttt{a}'s is $\ISA[p_s] = 3$ in Cases~(1) and~(3), and $\ISA[p_s] = 2$ in Case~(2)
because $k^{-1} = 5 \Leftrightarrow k \cdot k^{-1} \mod n = 1 \mod n$.
This also determines the constant~$m$ used in Lemma~\ref{lemRotate}.
Finally, we can fix the parameter~$t$ in Lemma~\ref{lemTernaryBWT} defined such that $p_t = p_1 - 1 - k \mod n$:
For that, write $\ISA[i] = (i - p_n) k^{-1} \mod n = (i + k - p_1) k^{-1} \mod n$ and compute
$\ISA[p_t] = \ISA[p_1-1-k] = (-1) k^{-1} \mod n = n - k^{-1} \mod n$.

\section{Applications}\label{secApplications}
We conclude our main results of \cref{secAPSA} by drawing connections between strings having arithmetically progressed suffix arrays
and Christoffel words (\cref{secChristoffel}), balanced words (\cref{secBalanced}), and Fibonacci words (\cref{secFibonacci}).

\subsection{Christoffel Words}\label{secChristoffel}
Christoffel words are binary strings whose origins are considered to date from Bernoulli's 1771 work \cite{B71}. 
Christoffel words can be described geometrically in terms of a line segment and associated polygon traversal
\cite{BL93}:
let $(p, q) \in \mathbb{N}^2$ where $(p, q)$ are coprime 
and let $\mathcal{S}$ be the line segment with endpoints $(0, 0)$ and $(p, q)$. 
The \Define{induced path} of a binary string~\word{} is a list of points $v_0, \ldots, v_n \in \mathbb{N}^2$ such that
$v_0 = (0,0)$, $v_n = (p,q)$, and for each $i \in [1..n]$, $v_{i} - v_{i-1} = (1,0)$ if $\word[i]= \texttt{a}$ and $v_{i} - v_{i-1} = (0,1)$ if $\word[i] = \texttt{b}$.
The string $\word{} \in \{\texttt{a},\texttt{b} \}^*$ is a \Define{lower Christoffel} word if the path induced by \word{} from the origin 
to $(p, q)$ is below the line segment $\mathcal{S}$ and the path and $\mathcal{S}$ determines a simple polygon which contains no other point in $\mathbb{N}^2$. 
An \Define{upper Christoffel} word is defined analogously by taking the path above $\mathcal{S}$. 
Hence, a Christoffel word is defined by a direction (above or below) and the \Define{slope} $p/q$, which determines $p$ and $q$ since $p$ and $q$ are coprime.
For instance, Case (3) of \cref{figBinaryExample} determines a Christoffel word with slope $\frac{5}{3}$. It follows from the coprimality of $p$ and $q$ that Christoffel words are necessarily primitive.
In what follows, we focus on lower Christoffel words, and drop the \emph{lower} adjective when speaking of (lower) Christoffel words.

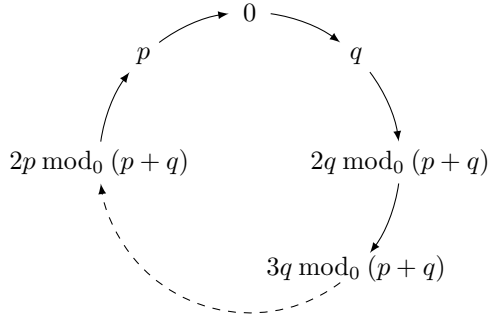
\begin{figure}[t]
\centering
    \begin{tikzpicture}[xscale=-1]
        \def \n {8}
        \def \radius {2cm}
        \def \margin {8} 

        \foreach \s in {1,2,3,7,\n}
        {
          \draw[->, >=latex] ({360/\n * (\s - 1)+\margin+90}:\radius) 
            arc ({360/\n * (\s - 1)+\margin+90}:{360/\n * (\s)-\margin+90}:\radius);
        }
        
        \node at ({360/\n * (0)+90}:\radius) {$0$};
        \node at ({360/\n * (1)+90}:\radius) {$q$};
        \node at ({360/\n * (2)+90}:\radius) {$2q \bmodz (p+q)$};
        \node at ({360/\n * (3)+90}:\radius) {$3q \bmodz (p+q)$};
        \node at ({360/\n * (7)+90}:\radius) {$p$};
        \node at ({360/\n * (6)+90}:\radius) {$2p \bmodz (p+q)$};
        
        \draw[->, >=latex,dashed] ({360/\n * (3)+\margin+90}:\radius) 
        arc ({360/\n * (3)+\margin+90}:{360/\n * (6)-\margin+90}:\radius);
    \end{tikzpicture}
	\caption{Cayley graph of $\mathbb{Z}/(p+q)\mathbb{Z}$ generated by $q$. See \cite[Fig.~1.4]{BLRS08} for a concrete example.}
	\label{figCayleyGraph}
\end{figure}

To show that every lower Christoffel word has an arithmetically progressed suffix array, we use an alternative characterization of Christoffel words based on Cayley graphs~\cite[Def.~1.4]{BLRS08}.
Let again $p, q \in \mathbb{N}$ be coprime.
\cref{figCayleyGraph} is the Cayley graph of $\mathbb{Z}/(p+q)\mathbb{Z}$ generated by $q$.
Cayley graphs are always simple cycles since $q$ and $p+q$ are coprime.
In what follows, we use $\bmodz$ with $n \bmodz n = 0 \bmodz n = 0$ to match the definition in~\cite{BLRS08},
as opposed to $\bmod$ with $0 \bmod n = n \bmod n = n$ elsewhere.

An edge $s \to t$ in the Cayley graph has the label~$\mathtt{a}$ if $s < t$, otherwise it has the label $\mathtt{b}$.
Reading the edge labels, starting from node $0$, following the edges of the Cayley graph and stopping when reaching node $0$ again, yields the lower Christoffel word $\wordC$ parametrized by $p$ and $q$.
The $i$-th node in the Cayley graph ($0$ being the first node) is $(i-1)q \bmodz (p+q)$.
Hence the $i$-th character of $\wordC$ is
\begin{align}
    \wordC[i] = \left\{ \begin{array}{ll}
	\texttt{a}\quad & \text{if }\; (i-1)q \bmodz (p+q) \;<\; iq \bmodz (p+q)\text{,} \\
	\texttt{b} & \text{if }\; (i-1)q \bmodz (p+q) \;>\; iq \bmodz (p+q)\text{.}
    \end{array}\right.
    \label{eq:christoffel_word}
\end{align}

Given the suffix array $\SA_{\wordC}$ of a lower Christoffel word $\wordC$, the split index $s$ (defined in \cref{thmBinaryStringCharacteristic}) is given by $p$, the total number of units along the $x$-axis in the polygonal path. 
All lower Christoffel words are Lyndon words \cite{BL93} and so necessarily border-free. Hence $\SA_{\wordC}[p_1] = 1$, and
$\SA_{\wordC}[s+1] = n$ since the string~$\wordC$ must end \texttt{b}.
The following theorem gives now the connection between lower Christoffel words and the strings of \cref{secAPSA}:

\begin{theorem}\label{thmChristoffelProgression}
    Let $p, q \in \mathbb{N}$ be coprime.
    Then the lower Christoffel word $\wordC$ characterized by $p$ and $q$ has an arithmetically progressed suffix array.
    The suffix array is given by the arithmetic progression $P$ with $p_1 = 1$ and $k = q^{-1}$, where $q^{-1}$ is the multiplicative inverse of $q$ in $\mathbb{Z}/n\mathbb{Z}$.
    The string~$\wordC$ is identical to Case~3 of \cref{thmBinaryStringCharacteristic} characterizing the binary case.
\end{theorem}
\begin{proof}
    We proof the theorem by showing that the Christoffel word $\wordC$ is equal to the string described in \cref{thmBinaryStringCharacteristic} as Case~3 
    whose suffix array is the arithmetically progressed array~$P$.

    Let $n = p+q$.
    By \cref{eq:christoffel_word} the $i$-th character of $\wordC$ is an $\mathtt{a}$ if and only if $(i-1)q \bmodz n \;<\; iq \bmodz n$.
    We can rewrite that to $(i-1)q \bmodz n \;<\; (i-1)q \bmodz n + q \bmodz n$.
    This condition is fulfilled if and only if $((i-1)q \bmodz n) + q < n$.
    We can rewrite that to $(i-1)q \bmodz n < n - q$.
    Replacing $\bmodz$ with $\bmod$, we obtain
    \begin{align*}
        \wordC[i] = \left\{ \begin{array}{ll}
	    \texttt{a}\quad & \text{if }\; 1 + (i-1)q \bmod n \;<\; n + 1 - q\text{,} \\
	    \texttt{b} & \text{otherwise}\text{.}
        \end{array}\right.
    \end{align*}
    
    Let $k = q^{-1}$.
    Let $\wordS$ denote the string of Case~3 in \cref{thmBinaryStringCharacteristic}, which is characterized by $p_1 = 1$ and $p_s = n - k = n - q^{-1}$.
    Using the results from \cref{secInversePermutation}, \cref{eqInducedBinaryInverse}, $\wordS$ can be written as
    \begin{equation}\label{eqCayleyA}
        \wordS[i] = \left\{ \begin{array}{ll}
	    \texttt{a}\quad & \text{if }\; (i + k - 1)k^{-1} \bmod n \;\le\; s\text{,} \\
	    \texttt{b} & \text{otherwise}\text{.}
        \end{array} \right.
    \end{equation}
    Substituting $k = q^{-1}$, $k^{-1} = q$ and $s = \ISA[p_s] = \ISA[n - q^{-1}] = n - q$, we can rewrite the condition for $\mathtt{a}$ in \cref{eqCayleyA} as $1 + (i-1)q \bmod n \le n - q$.
    Replacing $\le$ with $<$ we obtain the same definition for $\wordS$ as for $\wordC$, concluding the proof.
    \begin{align*}
        \wordS[i] = \left\{ \begin{array}{ll}
	    \texttt{a}\quad & \text{if }\; 1 + (i-1)q \bmod n \;<\; n + 1 - q\text{,} \\
	    \texttt{b} & \text{otherwise}\text{.}
        \end{array}\right.
    \end{align*}
\end{proof}

\begin{figure}
  \setlength{\tabcolsep}{0.1em}
\begin{tabular}{l
@{\hskip 0.55em}
*{12}{l}
@{\hskip 0.55em}
r*{12}{c}l
@{\hskip 0.55em}
c@{\hskip 0.55em}c}
		\toprule
\multicolumn{12}{c}{\word{}} & \multicolumn{14}{c}{\SA{}} &~ $p_s$ & $s$ \\
		\cmidrule(r){2-13}
		\cmidrule{14-26}
		&
		{\scriptsize 1} &
		{\scriptsize 2} &
		{\scriptsize 3} &
		{\scriptsize 4} &
		{\scriptsize 5} &
		{\scriptsize 6} &
		{\scriptsize 7} &
		{\scriptsize 8} &
           {\scriptsize 9} &
           {\scriptsize 10} &
           {\scriptsize 11} &
           {\scriptsize 12} &
		  &
		{\scriptsize 1} &
		{\scriptsize 2} &
		{\scriptsize 3} &
		{\scriptsize 4} &
		{\scriptsize 5} &
		{\scriptsize 6} &
		{\scriptsize 7} &
		{\scriptsize 8} &
           {\scriptsize 9} &
           {\scriptsize 10} &
           {\scriptsize 11} &
           {\scriptsize 12} &
		  &
		\\
		(1) & 
		\texttt{a} & 
		\texttt{a} & 
		\texttt{b} & 
	     \texttt{a} & 
		\texttt{b} & 
		\texttt{a} & 
		\texttt{a} & 
		\texttt{b} & 
		\texttt{a} & 
          \texttt{b} & 
	     \texttt{a} & 
		\texttt{b} &
	[ &
		1, &
		6, &
		11, &
		~4, &
		9, &
		2, &
		7, &
		12, &
           ~5, &
          10, &     
          ~ 3, & 
          ~ 8
		] &
		~~~7 & 7 \\
(2) &
	\texttt{b} & 
	\texttt{a} & 
     \texttt{b} & 
	\texttt{a} & 
	\texttt{b} & 
     \texttt{a} &
	\texttt{a} & 
	\texttt{b} & 
	\texttt{a} & 
	\texttt{b} & 
	\texttt{a} & 
	\texttt{a} & 
		[ &
		1, &
		8, &
		3, &
		10, &
		5, &
		12, &
		7, &
		2, &
          ~9, &
		~4, &
		~ 11, &
           ~~6
		] &
		~~~5 & 5 \\
		\bottomrule
	\end{tabular}
	\caption{Two Christoffel words with $S = ((0,0), (7,5))$, slope $\frac{5}{7}$, and length 12. The suffix array of the the lower Christoffel word (1) has an arithmetically progressed permutation with ratio~$k=5$, $p_s = 7 = n-k$, and hence is an instance of \cref{thmBinaryStringCharacteristic}, Case 3. 
Then, assuming $\texttt{b} < \texttt{a}$, the suffix array of the upper Christoffel word (2) has an arithmetic progression ratio of $n-k = 7$ with split index $n-s$ and position of largest suffix starting
$\texttt{b}$ of $n-p_s$.
	}
	\label{figChristoffelExample}
\end{figure}

In the geometric setting, traversing the path of the polygon from $(0,0)$ to $(p, q)$ is equivalent to scanning the characters in the defining Christoffel word $\wordC$, hence we can associate each character with its polygon coordinates. 
The BWT can be obtained from iterating across the suffix array and for each index selecting the preceding character in the text -- the BWT for the first string in Figure \ref{figChristoffelExample} has maximal runs in the form $\texttt{b}^5 \texttt{a}^7$; more generally, the BWT of a Christoffel word over $\Sigma = \{\texttt{a},\texttt{b}\}$ with slope $\frac{q}{p}$ takes the form $\texttt{b}^q \texttt{a}^p$ \cite{BLRS08}. Further, since the BWT is injective on Lyndon
words (see \cite{BLRS08}) it follows that this property holds for Christoffel words.\\

The progression ratio $q^{-1}$ of \cref{thmChristoffelProgression} may be useful for accessing geometries of interest in the polygon or aiding discovery of geometric repetitive structures.
For instance, to access the `steps' in the polygon in decreasing width, start at the origin and apply increments of $k$: for the first string in Figure \ref{figChristoffelExample} the widest steps are at coordinates $(0,0)$ and $(3,2)$ - see Figure \ref{fig_Christ}.
Note that the geometric width of this polygon at a certain height reflects a run of the character~\texttt{a} in its representing Christoffel word~$\wordC$. Hence, our question of steps in decreasing width can be translated to finding text positions $\wordC[i]$ with $\wordC[i..]$ having a long common prefix with $\texttt{a}\cdots$.

We 
now extend a known result for lower Christoffel words \cite[Sect.~6.1]{BLRS08}, 
which distinguishes consecutive rows in an associated BWT matrix:

\begin{lemma}
Let $\wordC$ be a 
Christoffel word of length $n \ge 2$ over $\Sigma = \{\mathtt{a},\mathtt{b} \}$. Suppose $\wordC$ has an arithmetically progressed suffix array with progression ratio $k$. Then in the BWT matrix~$\mathbf{M}$ of $\wordC$
with $\mathbf{M}[i,j] = \wordC[(\SA_{\wordC}[i]+j-1) \mod n]$,
rows $i$ and $i+1$ for $1 \le i < n$ differ in exactly two consecutive positions. For each row $i$, these positions are at $\mathbf{M}[i, i(n-k) \bmod n]$ and $\mathbf{M}[i, i(n-k)+1 \bmod n]$. 
\end{lemma}

\begin{proof}
The Christoffel word $\wordC$ is characterized by \cref{thmBinaryStringCharacteristic} (3). Further, since $\wordC$ is also a Lyndon word then it starts with \texttt{a} and ends with \texttt{b}. The second row in the BWT matrix $\mathbf{M}$ of $\wordC$ is  
$\wordC [\SA_{\wordC}[2]..n] \wordC[1 .. \SA_{\wordC} [2]-1]$
where the prefix $\wordC [\SA_{\wordC}[2] .. n]$ has length $n-k$ and ends with \texttt{b}.

From Case~3 of \cref{thmBinaryStringCharacteristic}, $p_s = n-k$, and so $\wordC [n-k] = \texttt{a}$, and since $\wordC [n-k .. n]$ is the largest suffix starting with \texttt{a} and $n \ge 2$, then it has the prefix \texttt{a}\texttt{b}. 
So $\mathbf{M}[1, n-k] = \texttt{a}$ and $\mathbf{M}[1, n-k+1] = \texttt{b}$, 
while from $\wordC[1] = \texttt{a}$ and $\wordC[n] = \texttt{b}$, 
$\mathbf{M}[2, n-k] = \texttt{b}$ and $\mathbf{M}[2, n-k+1] = \texttt{a}$.
This gives an arrangement in the first two rows of $\mathcal{A}=$ 
$\frac{\texttt{a} \texttt{b}}{\texttt{b} \texttt{a}}$ starting at position $n-k$ in row 1 and then at position 
$2(n-2) \bmod n$ in row 2. Similarly for each subsequent row $i$ in $\mathbf{M}$, $i<n$, there is an additional factor of $(n-k) \bmod n$ for the position of occurrence of $\mathcal{A}$.

We proceed to show that there are no other differences between adjacent rows other than those occurring in $\mathcal{A}$. So suppose instead that $\mathbf{M}[i, j] = \texttt{c}$ and $\mathbf{M}[i+1, j] = \texttt{d}$ for two distinct characters $\texttt{c},\texttt{d} \in \Sigma$, $1 \le i <n$ and $i(n-k)+1 < j < i(n-k) \bmod n$. Then the character~\texttt{d} in row $i+1$ is the same character in $\wordC$~as the one at position $j-(n-k) \bmod n$ in row $i$ which must be the same character as the one at position $j-(n-k) + (n-k) \bmod n$ in row $i$, namely a $\texttt{c}$, contradicting the claim.

Hence the only differences between adjacent rows occur at $\mathcal{A}$ which coincide with $p_s$ and the end of $\wordC$. In particular, $\mathcal{A}$ determines the lexicographic ordering of the rows of the matrix $\mathbf{M}$.
\end{proof}

The arithmetic progression of 
arithmetically progressed suffix arrays of Christoffel words 
allows us to determine the following factorizations of such words in constant time (without looking at the explicit characters of the word):

\begin{itemize}
\item \Define{Right factorization} \cite{BLRS08}; originally called the \Define{standard factorization} \cite{CFL58,DBLP:journals/jal/Duval83}. If $\s{w} = \s{uv}$ is a Lyndon word with \s{v} its lexicographically least proper suffix, then \s{u} and \s{v} are also Lyndon words and $\s{u} < \s{v}$.
Equivalently, the right factor \s{v} of the standard factorization can be defined to be the longest proper suffix of \s{w} that is a Lyndon word~\cite{bassino05standard}.

\item \Define{Left factorization} \cite{S62,V78}. If $\s{w} = \s{uv}$ is a Lyndon word with \s{u} a proper Lyndon prefix of maximal length, then \s{v} is also a Lyndon word and $\s{u} < \s{v}$.

\item 
\Define{$\BalancedTwo$ factorization} \cite{Mel1999}. A Lyndon word \s{w} is $\BalancedTwo$ if \s{w} is a character or there is a factorization $\s{w} = (\s{u}, \s{v})$ that is simultaneously the left and the right factorization of \s{w}, and \s{u} and \s{v} are also $\BalancedTwo$.
\end{itemize}

For any of the three above factorizations $\s{w} = (\s{u}, \s{v})$, we will say that the \Define{factorization index} is the index of the last character of \s{u} in \s{w}. 
The factorization index determines the split position and hence the factorization.

\begin{lemma}
\label{lem-balanced}
Let \wordC{} be a
Christoffel word of length $n \ge 3$ over $\Sigma = \{\mathtt{a},\mathtt{b} \}$.
Suppose $\wordC$ has an arithmetically progressed suffix array $\SA_{\wordC}$ with progression ratio $k \ge 1$ and split index $s$. 
Then $\wordC$ is a $\BalancedTwo$ Lyndon word with a factorization $\wordC = \word_{\textup{u}} \word_{\textup{v}}$
such that $|\word_{\textup{u}}| = \SA_{\wordC}[(s+2) \bmod n]$.
For $n=1,2$ we have that~$\wordC{} \in \{\mathtt{a},\mathtt{b}, \mathtt{ab}\}$, $k=1$, and that the split index~$s$ is 1.
\end{lemma}
\begin{proof}
  We apply the result that a word is a lower Christoffel word if and only if it is a $\BalancedTwo$ Lyndon word \cite{Mel1999} or \cite[Theorem~6.7]{BLRS08}. 
  So given \wordC{},
it remains to prove that 
$|\word_{\text{u}}|= \SA_{\wordC}[(s+2) \bmod~n]$.
Let $j \in [1..n]$ be such that $|\word_{\text{u}}| =  \SA_{\wordC}[j]$.
Then $\word_{\text{u}} = \wordC[1..\SA_{\wordC}[j]]$ and $\word_{\text{v}} = \wordC[\SA_{\wordC}[j]+1..n]$. 
Since $\word_{\text{u}}$ and $ \word_{\text{v}}$
are Lyndon words, we have that
\begin{itemize}
  \item $j > s$, otherwise $\wordC[\SA_{\wordC}[j]] = \texttt{a}$ and $\word_{\text{u}}$ would have a border.
  \item $\SA_{\wordC}[(s+1) \bmod~n] = n$ with $\wordC[n] = \texttt{b}$ since $\wordC$ is a Lyndon word larger than one, and the smallest suffix starting with \texttt{b} is $\wordC[n..n]$.
  \item The second smallest suffix starting with \texttt{b} is $\wordC[\SA_{\wordC}[2]-1..n]$.
  To see that observe that, because $\wordC$ is a Lyndon word, $\SA_{\wordC}[1] = 1$, and hence $\SA_{\wordC}[2] \ge 2$.
  Next, $\wordC[\SA_{\wordC}[2]-1] = \texttt{b}$, otherwise $\wordC[\SA_{\wordC}[2]-1..n] \prec \wordC[\SA_{\wordC}[2]..n]$ would give a smaller suffix.
  Hence, $\SA_{\wordC}[(s+2) \bmod~n] = \SA_{\wordC}[2]-1$.
\end{itemize}
  Finally, $\word_{\text{v}} = \wordC[\SA_{\wordC}[2]..n]$ since otherwise we yield a contradiction to the right factorization that 
  $\word_{\text{v}}$ is the lexicographically least proper suffix.
  So $j = \SA_{\wordC}[(s+2) \bmod~n]$.
\end{proof}

For example, the Christoffel word in \cref{figChristoffelExample} Case (1) has factorization index $\SA_{\wordC}[9] = 5$ with the $\BalancedTwo$ factorization $(\texttt{a} \texttt{a} \texttt{b} \texttt{a} \texttt{b})(\texttt{a}  \texttt{a} \texttt{b}  \texttt{a} \texttt{b}  \texttt{a}  \texttt{b})$; the first level of recursion gives the $\BalancedTwo$ factorizations  $(\texttt{a} \texttt{a} \texttt{b})(\texttt{a} \texttt{b})$ and $(\texttt{a}  \texttt{a} \texttt{b}  \texttt{a} \texttt{b})(\texttt{a}  \texttt{b})$.

\begin{figure}[H]
\centering
\begin{tikzpicture}[x=-1cm, every node/.style={inner sep=0,outer sep=1mm, anchor=base, yshift=-3, fill=white}]
\newlength{\letterspace}
\setlength{\letterspace}{1mm}

\node at (0, -0.3) {$(0,0)$};

\node at (-7, 5.3) {$(7,5)$};

\node at (-3.3, 2.6) {$\mathcal{S}$};

\draw[thick] (0,0) -- (-7,5);

\node at (-4, 1.25) {$\word_{\text{L}}$};

\draw (0,0) -- (-2,0);

\node at (-0.5, -0.2) {$\texttt{a}$};

\node at (-1.5, -0.2) {$\texttt{a}$};

\draw (-2,0) -- (-2,1);

\node at (-2.2, 0.5) {$\texttt{b}$};

\draw (-2,1) -- (-3,1);

\node at (-2.5, 0.8) {$\texttt{a}$};

\draw (-3,1) -- (-3,2);

\node at (-3.2, 1.5) {$\texttt{b}$};

\draw (-3,2) -- (-4,2);

\node at (-3.5, 1.8) {$\texttt{a}$};

\draw (-4,2) -- (-5,2);

\node at (-4.5, 1.8) {$\texttt{a}$};

\draw (-5,2) -- (-5,3);

\node at (-5.25, 2.5) {$\texttt{b}$};

\draw (-5,3) -- (-6,3);

\node at (-5.5, 2.8) {$\texttt{a}$};

\draw (-6,3) -- (-6,4);

\node at (-6.25, 3.5) {$\texttt{b}$};

\draw (-6,4) -- (-7,4);

\node at (-6.5, 3.8) {$\texttt{a}$};

\draw (-7,4) -- (-7,5);

\node at (-7.25, 4.5) {$\texttt{b}$};

\draw[dotted, thick] (0,0) -- (-3,2);

\draw[dotted, thick] (-3,2) -- (-7,5);

\node at (-3, 3.75) {$\word_{\text{U}}$};

\draw (0,0) -- (0,1);

\node at (0.2, 0.5) {$\texttt{b}$};

\draw (0,1) -- (-1,1);

\node at (-0.5, 1.2) {$\texttt{a}$};

\draw (-1,1) -- (-1,2);

\node at (-0.8, 1.5) {$\texttt{b}$};

\draw (-1,2) -- (-2,2);

\node at (-1.5, 2.2) {$\texttt{a}$};

\draw (-2,2) -- (-2,3);

\node at (-1.8, 2.5) {$\texttt{b}$};

\draw (-2,3) -- (-4,3);

\node at (-2.5, 3.2) {$\texttt{a}$};

\node at (-3.5, 3.2) {$\texttt{a}$};

\draw (-4,3) -- (-4,4);

\node at (-3.8, 3.5) {$\texttt{b}$};

\draw (-4,4) -- (-5,4);

\node at (-4.5, 4.2) {$\texttt{a}$};

\draw (-5,4) -- (-5,5);

\node at (-4.8, 4.5) {$\texttt{b}$};

\draw (-5,5) -- (-7,5);

\node at (-5.5, 5.2) {$\texttt{a}$};

\draw (-6,5) -- (-7,5);

\node at (-6.5, 5.2) {$\texttt{a}$};

\end{tikzpicture}

\caption{\label{fig_Christ} An upper Christoffel word $\word_{\text{U}} =  \texttt{b} \texttt{a} \texttt{b} \texttt{a} \texttt{b} \texttt{a} \texttt{a} \texttt{b} \texttt{a} \texttt{b} \texttt{a} \texttt{a}$ and its reversal, the lower Christoffel word $\word_{\text{L}} = \texttt{a} \texttt{a} \texttt{b} \texttt{a} \texttt{b} \texttt{a} \texttt{a} \texttt{b} \texttt{a} \texttt{b} \texttt{a} \texttt{b}$, shown geometrically with respect to the line segment $\mathcal{S} = ((0,0),(7,5))$. The dotted lines indicate the 
$\protect\BalancedTwo$ factorization of $\word_{\text{L}}$ into the factors $\texttt{a} \texttt{a} \texttt{b} \texttt{a} \texttt{b}$ and $\texttt{a} \texttt{a} \texttt{b} \texttt{a} \texttt{b} \texttt{a} \texttt{b}$.}
\end{figure}

Conditions for the factorization of a Lyndon word into exactly two non-overlapping Lyndon factors are given in \cite{DBLP:journals/fuin/DaykinDS09}, where \emph{overlapping factors} have non-empty suffixes and prefixes sharing the same characters.
We can geometrically consider the $\BalancedTwo$ factorization of Christoffel words as follows:

\begin{lemma}
\label{lemmEuclidean}
Let $p, q \in \mathbb{N}$ be coprime and $\wordC$ be the lower Christoffel word  characterized by $p$ and $q$ with factorization index $i$. Further, let the line segment $\mathcal{S} = ((0,0),(p,q))$ be associated with $\wordC$. Then for all points on the path determined by $\wordC$ (apart from the end points), $\wordC[i]$ has the shortest Euclidean distance to $\mathcal{S}$.
\end{lemma}
\begin{proof}
Since $\wordC$ is a Christoffel word it is also a $\BalancedTwo$ Lyndon word. Let the $\BalancedTwo$ factorization be $\wordC = \word_{\text{u}} \word_{\text{v}}$, where the factors have defining line segments $\mathcal{S}_{\text{u}}$, $\mathcal{S}_{\text{v}}$ respectively, i.e., $\mathcal{S}_{\text{u}}$ is the line from $(0,0)$ to the point~$s$ being the geometric representation of $\wordC[1..|\word_{\text{u}}|]$, and $\mathcal{S}_{\text{v}}$ is the line from $s$ to~$(p,q)$, 
cf.\ \cref{fig_Christ} with $\mathcal{S}_{\text{u}} = \texttt{aabab}$ and $s = (3,2)$.
Now assume that there exists a point~$r$ that geometrically represents a prefix $\wordC[1..i]$ with the property that $r$ does not have the shortest Euclidean distance to $\mathcal{S}$. Then at least one of the paths for $\word_{\text{u}}$ and $\word_{\text{v}}$ must cross their associated line $\mathcal{S}_{\text{u}}$ or $\mathcal{S}_{\text{v}}$, contradicting the geometric definition of a Christoffel word.  
\end{proof}

  \subsection{Balanced Words}\label{secBalanced}
	A binary string~\word{} is called a \Define{balanced word} if for each character~$c \in \{\texttt{a},\texttt{b}\}$, the number of occurrences of $c$ in $\s{U}$ and in $\s{V}$ differ by at most one, for all pairs of substrings~\s{U} and \s{V} with $|\s{U}| = |\s{V}|$ of the infinite concatenation $\word \cdot \word \cdots$ of~$\word$.

	\begin{lemma}
	  Let $\word$ be a string over the binary alphabet $\{\mathtt{a}, \mathtt{b}\}$ with an arithmetically progressed suffix array $P = [p_1, \dots, p_n]$ with ratio $k$.
		Given that $p_1 \neq k-1$, $\word$ is balanced.
	\end{lemma}
	\begin{proof}
		According to \cref{thmMatrixBWT} for $\word$ the $\BWT$ defined on the suffix array and the $\BWT$ defined on the BWT matrix are identical.
		From \cref{corSimpleBWT} we know that the $\BWT$ of $\word$ has the shape $\texttt{b}^x \texttt{a}^{y}$.
		By~\cite[Thm.~2]{restivo2009balanced} for binary words the following conditions are equivalent:
		\begin{enumerate}
			\item There are two coprime numbers~$p$ and $q$ such that $\BWT_\word = \texttt{b}^p\texttt{a}^q$; \item \word{} is balanced.
		\end{enumerate}
		We conclude the proof by showing that $x$ and $y$ are co-prime.
		Since $p_1 \neq k-1$ we have $p_1 \in \{1, n\}$.
		Using the results from \cref{secInversePermutation} the number of \texttt{a}'s in \word{} is $n - k^{-1}$.
		Thus the number of \texttt{b}'s is $k^{-1}$.
		As $k^{-1}$ is co-prime to $n$, $n - k^{-1}$ must be co-prime to $k^{-1}$.
	\end{proof}

\subsection{Relation to the Fibonacci word sequence}\label{secFibonacci}
The Fibonacci sequence is a sequence of binary strings~$\{\s{F}_m\}_{m \ge 1}$
with  $\s{F}_1 :=\texttt{b}$, $\s{F}_2 := \texttt{a}$, and $\s{F}_m := \s{F}_{m-1} \s{F}_{m-2}$.
Then
$\s{F}_m$ and $f_m := |\s{F}_m|$ are called the $m$-th \Define{Fibonacci word} and the $m$-th \Define{Fibonacci number}, respectively.

{K\"{o}ppl and I}~\cite[Thm.~1]{koppl15fibonacci} observed that
the suffix array of $\s{F}_m$ for even~$m$ is the arithmetically progressed permutation~$\SA_{\s{F}_m}$ with ratio~$f_{m-2}$ 
mod $f_m$ and $\SA_{\s{F}_m}[1] = f_m$.
\Cref{thmBinaryStringCharacteristic} generalizes this observation by characterizing all binary strings whose suffix arrays are arithmetically progressed.
Hence, $\s{F}_m$ must coincide with Case~1 of \cref{thmBinaryStringCharacteristic} since it ends with character~\texttt{a}.

\begin{lemma}
	The Fibonacci word $\s{F}_m$ for even~$m$ is given by
	\begin{equation*}
		\s{F}_m[i] =
		\begin{cases}
			\mathtt{a} & \text{~if~} 1 + i \cdot f_{m-2} \bmod f_m \le f_{m-1} \text{, or~}\\
			\mathtt{b} & \text{~otherwise.}
		\end{cases}
	\end{equation*}
\end{lemma}
\begin{proof}
	We use the following facts:
\begin{itemize}
\item The greatest common divisor of $f_i$ and $f_j$ is the Fibonacci number whose index is the greatest common divisor of~$i$ and~$j$~\cite[Fibonacci numbers]{wells05prime}.
Hence, $f_{m-1}$ and $f_{m}$ are coprime for every~$m \ge 2$.
\item $f^2_{m-2} \mod f_m = 1$ holds for every even~$m \ge 3$~\cite{hoggatt77fibonacci}.
	Hence, $k^{-1} = k = f_{m-2}$.
\item By definition, $\s{F}_m[f_m] = \texttt{a}$ if $m$ is even, and therefore $\SA_{\s{F}_m}[1] = f_m$.
\end{itemize}
The split position~$p_s$ is $p_s = f_m - k = f_{m-1}$.
So $\SA_{\s{F}_m}[f_m] = f_m - k = f_m - f_{m-2}$.
By \cref{thmInversePermutation}, 
$\ISA_{\s{F}_m}[i] 
=  i f_{m-2} - \SA_{\s{F}_m}[f_m] f_{m-2} \mod f_m
= i f_{m-2} + 1 \mod f_m $,
where 
$- \SA_{\s{F}_m}[f_m] f_{m-2} \mod f_m = 
(f_{m-2} - f_m)  f_{m-2} \mod f_m =
1 - f_m f_{m-2} \mod f_m = 
1$.
The rest follows from \cref{eqInducedBinaryInverse}.
\end{proof}

\newcommand*{\Rabbit}[1]{\ensuremath{\s{\bar{F}}_{#1}}}
Let $\Rabbit{m}$ denote the $m$-th Fibonacci word whose \texttt{a}'s and \texttt{b}'s are exchanged, 
i.e., $\Rabbit{m} = \texttt{a} \Leftrightarrow \s{F}_m = \texttt{b}$.
\begin{lemma}
	$\SA_{\Rabbit{m}}$ is arithmetically progressed with ratio~$f_{m-2}$ for odd~$m$.
\end{lemma}
\begin{proof}
	Since $\Rabbit{m}[|\Rabbit{m}|] = \texttt{a}$ for odd~$m$, 
	$\Rabbit{m}[f_m..]$ is the lexicographically smallest suffix.
	Hence, $\SA := \SA_{\Rabbit{m}} = |\Rabbit{m}|$.
	If $\SA$ is arithmetically progressed with ratio~$k$, then its split position must be $p_s = n-k-1$ according to~\cref{thmBinaryStringCharacteristic}.
We show that $k = f_{m-2}$ by proving
\begin{align*}
\Rabbit{m}[f_m ..] \prec \Rabbit{m}[f_m + f_{m-2} \bmod f_m .. ] 
&\prec  \Rabbit{m}[f_m + 2f_{m-2} \bmod f_m .. ] 
\prec \dots  \\
&\prec \Rabbit{m}[f_m + (f_m - 1)f_{m-2} \bmod f_m ..]
\end{align*}
 in a way similar to~\cite[Lemma~8]{koppl15fibonacci}.
 For that, let \s{\bar{S}} of a binary string~$\s{S} \in \{\texttt{a},\texttt{b}\}^*$ denote \s{S} after exchanging \texttt{a}'s and \texttt{b}'s (i.e., $\s{\bar{S}} = \texttt{a} \Leftrightarrow \s{S} = \texttt{b}$).
 Further, let $\lessdot$ be the relation on strings such that $\s{S} \lessdot \s{T}$ if and only if $\s{S} \prec \s{T}$ and \s{S} is \emph{not} a prefix of \s{T}.
 We need this relation since $\s{S} \lessdot \s{T} \Longleftrightarrow \s{\bar{S}} \gtrdot \s{\bar{T}}$ while
 $\s{S} \prec \s{T}$ and $\s{\bar{S}} \prec \s{\bar{T}}$ holds if \s{S} is a prefix of \s{T}.

	\begin{itemize}
		\item For $i \in [1..f_{m-1})$, we have $\s{F}_m[i..] \gtrdot \s{F}_m[i + f_{m-2}..]$ due to~\cite[Lemma~7]{koppl15fibonacci}, 
			thus $\Rabbit{m}[i..] \lessdot \Rabbit{m}[i + f_{m-2}..]$.
		\item For $i \in (f_{m-1}..f_m]$, since 
			$\s{F}_m = \s{F}_{m-1} \s{F}_{m-2} = \s{F}_{m-2}\s{F}_{m-3} \s{F}_{m-2}$, 
			$\s{F}_m[i..]$ is a prefix of $\s{F}_m[i - f_{m-1}..] = \s{F}_m[i + f_{m-2} \bmod f_m ..]$.
			Therefore, $\Rabbit{m}[i..]$ is a prefix $\Rabbit{m}[i + f_{m-2} \bmod f_m ..]$ and $\Rabbit{m}[i..] \prec \Rabbit{m}[i + f_{m-2} \bmod f_m ..]$.
\end{itemize}
		Since $f_m$ and $f_{m-2}$ are coprime, $\{ i + f_{m-2} \mod f_m \mid i > 0 \} = [1..n]$.
		Starting with the smallest suffix $\Rabbit{m}[f_m..]$, 
		we end up at the largest suffix $\Rabbit{m}[f_m+(f_m-1)f_{m-2} \bmod f_m..]$ after $m-1$ arithmetic progression steps of the form
		$\Rabbit{m}[f_m+ i f_{m-2} \bmod f_m..]$ for $i \in[0..f_m-1]$.
		By using one of the two above items we can show that these arithmetic progression steps yield a list of suffixes sorted in lexicographically ascending order.
\end{proof}

\section{Extension to Larger Alphabets}
	\label{secLargerAlphabets}
	In this section we extend our results of \cref{secAPSA} to alphabets of arbitrary size.
	Let $P = [p_1, \dots, p_n]$ be an arithmetically progressed permutation with ratio $k$.
	Let $\Sigma = \{1, 2, \dots, \sigma\}$ be an alphabet of size $\sigma$ where the order is given by $1 < 2 < \dots < \sigma$.
	To construct a string $\word$ over the alphabet $\Sigma$ having $P$ as its suffix array we proceed similarly to the construction presented in the previous sections:
	First we split $P$ into subarrays $S_1, \dots, S_\sigma$, then for each subarray $S_i$ we assign the character $i$ to each position $p \in S_i$.
	When splitting $P$ into subarrays there are some fixed positions where we are required to split $P$, while the remaining splitting positions can be chosen freely.
	Let $\sigma_{\text{min}}$ be the size of the minimal alphabet over which there is a string having $P$ as its suffix array.
	Then there are $\sigma_{\text{min}} - 1$ positions where we are required to split $P$.
	Those required splitting positions are after the following entries, modulo $n$:
	\begin{align*}
		&\{p_1 - k - 1, n - k\} && \text{if } p_1 \notin \{1, k + 1, n\}\text{,}\\
		&\{n - k\} && \text{if } p_1 \in \{1, k+1\}\text{,}\\
		&\{p_1 - k - 1\} && \text{if }p_1 = n\text{ and }k\neq n-1\text{.}
	\end{align*}

	\begin{example}
		Consider the alphabet $\Sigma = \{\texttt{a}, \texttt{b}, \texttt{c}, \texttt{d}, \texttt{e}\}$ with $\sigma := \vert \Sigma\vert = 5$ and order $\texttt{a}<\texttt{b}<\texttt{c}<\texttt{d}<\texttt{e}$, and the permutation $P = \lbrack 5, 2, 7, 4, 1, 6, 3, 8\rbrack$ which has length $n = 8$, is arithmetically progressed with ratio $k = 5$ and starts with the entry $p_1 = 5$.
		The minimum alphabet size to construct a string having $P$ as suffix array is $\sigma_{\text{min}} = 3$.
		Thus there are two required splitting positions.
		Those are after $3$ and $7$.
		Using vertical bars $|$ to denote the splitting positions we obtain $P = \lbrack 5, 2, 7| 4, 1, 6, 3| 8\rbrack$.
		There are $\sigma - \sigma_{\text{min}}$ splitting positions left which can be chosen freely.
		Assume we choose to split after entries $1$ and $2$.
		This leads to $P = \lbrack 5, 2| 7| 4, 1| 6, 3| 8\rbrack$.
		Assigning characters we obtain the string $\word = \texttt{cadcadbe}$.
	\end{example}

	At the beginning of this paper we have looked at the strings over an alphabet of size $\sigma$ having the suffix array~$[n,n-1,\ldots,1]$.
	The number of those strings is given by \cref{thmUnaryEnumeration} and is bounded by a polynomial in $n$ and~$\sigma$.
	We extend upon that result:
	For an arbitrary permutation, which is not necessarily arithmetically progressed we give a bound on the number of strings with that permutation as suffix array.
	For a fixed arithmetically progressed permutation~$P$ we give an exact formula for the number of strings having $P$ as their suffix array.
	We conclude that in total the number of strings having an arithmetically progressed suffix array is bounded by a polynomial in $n$ and~$\sigma$.

	\begin{lemma}
		\label{lemEnumerationBound}
		Let $\Sigma$ be an alphabet of size $\sigma = \vert \Sigma \vert$.
		Given a permutation $P$ of length $n$ there are at most $n + \sigma - 1 \choose n$ strings of length $n$ over the alphabet $\Sigma$ having $P$ as their suffix array.
	\end{lemma}

	\begin{proof}
		Let the alphabet $\Sigma$ be given by $\Sigma = \{1, \dots, \sigma\}$ with the order $1 < 2 < \dots < \sigma$.
		Let $\word$ be a string over $\Sigma$ of length $n$ whose suffix array is $P$.
		Permuting $\word$ by its suffix array $P$ we obtain the string $\word'$ with $\word'[i] = \word[P[i]]$.
		As $P$ is the suffix array of $\word$ we have $\word'[i..n] \prec \word'[j..n]$ and thus $\word'[i] \le \word'[j]$ for all $i < j$.
		Therefore $\word'$ is given by $\word' = 1^{t_1}\cdots \sigma^{t_\sigma}$, where $t_i$ describes the number of occurrences of $i$ in $\word'$.
		The problem of finding the number of strings with this particular shape can be reduced to the classic stars and bars problem~\cite[Chp.~II, Sect.~5]{feller68probability} with $n$ stars and $\sigma - 1$ bars, yielding $n + \sigma - 1 \choose n$ possible strings.
		As the permutation $P$ is a bijection on $\Sigma^n$ there is only one string $\word$ that, permuted by $P$, gives the string $\word'$.
		Thus there are at most $n + \sigma - 1 \choose n$ strings having $P$ as suffix array.
		There may be less as it is possible that a string $\Word$, which is a pre-image of some $\word' = 1^{t_1}\cdots \sigma^{t_\sigma}$ under the permutation $P$, has a suffix array different to $P$.
	\end{proof}

	\begin{lemma}
		\label{lemEnumerationGeneralized}
		Let $\Sigma$ be an alphabet of size $\sigma = \vert \Sigma \vert$.
		Given a permutation $P$ of length $n$, let $\sigma_{\text{min}}$ be the size of the smallest alphabet over which there exists a string having $P$ as its suffix array.
		Then there are at most $n + \sigma - 1 \choose \sigma - \sigma_{\text{min}}$ strings of length $n$ over the alphabet $\Sigma$ having $P$ as their suffix array.
	\end{lemma}
	\begin{proof}
		As described at the beginning of this section, all strings over the alphabet $\Sigma$ having $P$ as suffix array can be constructed by splitting $P$ into $\sigma - 1$ subarrays.
		There are $\sigma_{\text{min}} - 1$ positions where we are required to split $P$ and $\sigma - \sigma_{\text{min}}$ splitting positions that can be chosen freely.
		We model the selection of the positions that can be chosen freely using the stars and bars problem~\cite[Chp.~II, Sect.~5]{feller68probability}.
		The stars represent the entries of the permutation, the bars the splitting positions.
		Assume we start with $n + \sigma - 1$ stars.
		Then $\sigma - \sigma_{\text{min}}$ bars (splitting positions) can be freely chosen.
		There are $n + \sigma - 1 \choose \sigma - \sigma_{\text{min}}$ ways to do this.
		Then we replace the stars at the $\sigma_{\text{min}} - 1$ required splitting positions with bars.
		This gives us $n$ stars and $\sigma - 1$ bars, which we map to a partition of $P$ into $\sigma$ subarrays.
	\end{proof}

	\begin{lemma}
		\label{lemTotalBound}
		Let $\Sigma$ be an alphabet of size $\sigma = \vert \Sigma \vert$.
		Then the number of strings of length $n$ with an arithmetically progressed suffix array is at most $p(n,\sigma) = n(n-1){n + \sigma - 1 \choose \sigma - \sigma_{\text{min}}}$.
	\end{lemma}

	\begin{proof}
		Each arithmetically progressed permutation $P$ of length $n$ can be described by two parameters:
		Its first entry $P[1]$ a its ratio $k$.
		For $P[1]$ there are $n$ different options, for $k$ there are at most $n-1$ different options.
		Thus the number of arithmetically progressed permutations of length $n$ is at most $n(n-1)$.
		By \cref{lemEnumerationGeneralized} the number of strings of length $n$ having a specific permutation as suffix array is bounded by $n + \sigma - 1 \choose \sigma - \sigma_{\text{min}}$.
		Putting those two facts together we obtain that the number of strings of length $n$ with an arithmetically progressed suffix array is bounded by $p(n) = n(n-1){n + \sigma - 1 \choose \sigma - \sigma_{\text{min}}}$.
	\end{proof}
	
	Given a fixed alphabet $\Sigma$, by the above lemma, the number of strings of length $n$ with arithmetically progressed suffix array is bounded by a polynomial in $n$ and~$\sigma$.

	\section{Applications on Meta Strings}
	\label{secMeta}

	In this final section we overview some connections of suffix arrays with generalized forms of words and meta strings.
	A generalization of strings was proposed in the 1974 groundbreaking paper of Fischer and Paterson~\cite{FP74}, where string matching was considered in a more general setting than with the usual solid letters, whereby either string could have \emph{don't care}\footnote{Don't care symbols match with all characters.} symbols, and was achieved in time nearly as fast as linear.
	Uncertain sequences, including indeterminate strings, have application in inexact matching tasks, for instance, allowing for errors such as with Web interface data entry and Internet search.
	They are also useful for expressing DNA polymorphisms, that is biological sequence positions that can have multiple possibilities and encoded with IUPAC\footnote{International Union of Pure and Applied Chemistry} meta characters, for example \texttt{N} denotes any of the DNA nucleotides.
	A \emph{codon} is a form of meta character whereby a sequence of three nucleotides encodes a specific amino acid and are used for protein expression;
	so a genetic code can be composed of concatenated codon units.
	The \emph{truncated generalized suffix automaton} was introduced for indexing length-bounded $k$-factors of degenerate DNA and RNA sequences~\cite{FIRVV08}.
	An \emph{elastic-degenerate} string is a sequence of sets of strings used for succinctly representing a multiple alignment of a collection of closely related sequences (e.g. a pan-genome, that is all genes and genetic variation within a species),
and also supports approximate pattern matching~\cite{BPPR20}.
	Sequence alignment is useful for inferring evolutionary relationships between biological sequences.

	Daykin and Watson proposed a simple degenerate BWT, the $D$-BWT~\cite{DW17}, constructed by applying lex-extension order ({\it cf.} Example~\ref{exIndeterm}) to relabel the sets and order the degenerate strings in the $D$-BWT matrix.
	Subsequently in~\cite{DGGLLLMPW19}, the $D$-BWT was applied to pattern matching using the backward search technique.
This formalized and extended work implemented in~\cite{HPB13} presenting a bioinformatics software tool, BWBBLE, for pattern matching on a pan-genome that they called a reference multi-genome. \\

	In what follows, we first formally define indeterminate strings, and subsequently
	illustrate various approaches for defining an \emph{indeterminate suffix array}:
	An \Define{indeterminate string}\footnote{Also known as \emph{degenerate string} in the
	literature, particularly in the field of microbiology.} $\word{}^I = \word{}^I [1..n]$ on an alphabet $\Sigma$ is a sequence of nonempty subsets of $\Sigma$.
	Specifically, an indeterminate string $\word{}^I$ has the form $\word{}^I = t_1 \cdots t_n$ where each ${t_i}$ is a set of characters over $\Sigma$; 
	a singleton is known as a \Define{solid letter}.
	For example, 
	\[
	  \word{'} = \{\texttt{c},\texttt{a},\texttt{b},\texttt{e}\}
	\{\texttt{d}\}
	\{\texttt{c},\texttt{a},\texttt{b},\texttt{e}\}
	\{\texttt{d}\}
	\{\texttt{d}\}
	\{\texttt{c},\texttt{a},\texttt{b},\texttt{e},\texttt{f},\texttt{g}\}
      \]
      is a ternary border-free indeterminate string of length 6, with a singleton \{\texttt{d}\}.

	\begin{example}
		\label{exIndeterm}

		Let $\word{}^I$ = \{\texttt{b}\}\{\texttt{b}\}\{\texttt{c},\texttt{a}\}\{\texttt{b},\texttt{d},\texttt{a}\} be an indeterminate string, then $\SA_{\word{}^I}$ can be alternatively defined by the following approaches:

		\begin{itemize}

			\item We can apply the lex-extension order by treating each character set as a string whose characters are sorted, and use the lexicographic order of these strings to sort the sets, resulting in ranked meta characters:
			$\{\texttt{a,b,d}\} = A;~\{\texttt{a,c}\} = B;~\{\texttt{b}\} = C$.
			Then $\word{}^I$ = CCBA, and $\SA_{\word{}^I} = [4,3,2,1]$, an arithmetically progressed permutation.
			This method enables linear-time construction of an indeterminate suffix array.

			\item Given we do not want the input to be rearranged, then we can form the meta characters as follows: $\{\texttt{b}\} = A;~\{\texttt{b,d,a}\} = B;~\{\texttt{c,a}\} = C$.
			Here, $\word{}^I$ = AACB with suffix array $\SA_{\word{}^I} = [1,2,4,3]$, which is not arithmetically progressed.

			\item Finally, we can incorporate the suffix arrays of the individual sets treated as strings, so the above suffix array [1,2,4,3], with both indeterminate and solid letter positions, becomes
			  \[
			    [(1,1), (2,1), ((4,3),(4,1),(4,2)),((3,2),(3,1))].
			  \]

		\end{itemize}
	\end{example}

	Clearly, there is a natural generalization of many types of well-known patterned strings over solid letters to the indeterminate form, such as \emph{indeterminate Fibonacci words}, where the solid form can give a meta representation of the generalized form.

\section{Conclusion and Problems}\label{secConclusion}
Given an arithmetically progressed permutation~$P$ with ratio $k$, we studied the minimum alphabet size and the shape of those strings having~$P$ as their suffix array.
Only in the case $P = [n,n-1,\ldots,1]$, a unary alphabet suffices.
For general $P = [p_1,\ldots,p_n] \not= [n,n-1,\ldots,1]$, there is exactly one such string on the binary alphabet if and only if $p_1 \in \{1,k+1,n\}$.
In all other cases, there is exactly one such string on the ternary alphabet.
We conclude by proposing some research directions.
\begin{itemize}
\label{item-problems}

\item A natural question arising from this research is to characterize strings having arithmetic progression properties for the run length exponents of their BWTs, particularly for the bijective~\cite{gil12bbwt} or extended BWT~\cite{mantaci07ebwt}, which are always invertible.

For example, given the arithmetically progressed permutation 3214, 
then the run-length compressed string 
$\texttt{a}^3\texttt{c}^2 \texttt{\$} \texttt{b}^4$ 
(a) matches the permutation 3214 and (b) is a BWT image because its inverse is $\texttt{b}^2\texttt{cb}^2 \texttt{ca}^3 \texttt{\$}$, which can be computed by the Last-First mapping. 
However, for the same permutation, $\texttt{a}^3 \texttt{b}^2 \texttt{\$} \texttt{b}^4$ does not work since it is not a BWT image. Further examples of arithmetically progressed BWT exponents are: $\texttt{a}^3\texttt{c}^4 \texttt{\$} \texttt{b}^2$,  $\texttt{a}^4\texttt{c}^3\texttt{e}^2\texttt{\$}\texttt{d}^6\texttt{b}^5$, and $\texttt{a}^4\texttt{c}^3\texttt{e}^2\texttt{\$}\texttt{f}^7\texttt{d}^6\texttt{b}^5$.

\newcommand*{\LCP}{\textsf{LCP}}

\item Arithmetic properties can likewise be considered for the following stringology integer arrays: 
\begin{itemize}
	\item Firstly the \Define{longest common prefix} (LCP) array~\LCP{}, whose entry
	$\LCP[i]$ is the length of the longest common prefix of the \emph{lexicographically} $i$-th smallest suffix with its lexicographic predecessor for $i \in [2..n]$.
\item Given a string $\word \in \Sigma^+$ of length~$n$, 
the \Define{prefix table} $P_{\word}$ of \word{} is given by $P_{\word}[i] = \LCP(\word, \word[i..n])$ for $i \in [1..n]$; 
equivalently, the \Define{border table} $B_{\word}$ of \word{} is defined by 
\[
B_{\word}[i] = \max\{|\s{S}| \mid \s{S} \text{~is a border of~} \word[1..i]\} \text{~for~} i \in [1..n]. 
\]
\item Integer \Define{prefix lists} are more concise than prefix tables and give the lengths of overlapping LCPs of \word{} and suffixes of \word{} 
(cf.~\cite{clement17border}).

\item The $i$-th entry of the \Define{Lyndon array} $\s{\lambda} = \s{\lambda}_{\word}[1..n]$
of a given string~$\word = \word[1..n]$
is the length
of the longest Lyndon word that is a prefix of $T[i..]$ -- reverse engineering in~\cite{daykin18lyndon} includes a linear-time test
for whether an integer array is a Lyndon array.
Likewise, the \Define{Lyndon factorization array} $ \s{F} = \s{F}_{\word}[1..n]$ of $\word$ stores in its $i$-th entry
the size of the Lyndon factorization (i.e., the number of factors) of the suffix $ \s{\word}[i..n]$. 
The problems are to characterize those arithmetic progressions that define a valid Lyndon array, respectively Lyndon factorization array.
For example, consider the string $\word = \texttt{azyx}$, then its Lyndon array is $\s{\lambda}_{\word} = [4,1,1,1]$, while the Lyndon factorization array is $\s{F}_{\word} = [1,3,2,1]$.
Trivially, for $\word = \texttt{abc} \dots \texttt{z}$ the Lyndon array is an arithmetic progression and likewise for the Lyndon factorization array of 
$\word = \texttt{z}^t\texttt{y}^t\texttt{x}^t \ldots \texttt{a}^t$ for 
$\texttt{z} > \texttt{y} > \texttt{x} > \ldots > \texttt{a}$.
\end{itemize}

\item A challenging research direction is to consider arithmetic progressions for multi-dimensional suffix arrays and Fibonacci word sequences.

\end{itemize}

\subsection*{Acknowledgements}
We thank Gabriele Fici for the initial help in guiding this research started at StringMasters.

Funding: This research was part-funded by JSPS KAKENHI with grant number JP18F18120 and JP21K17701, and
by the European Regional Development Fund through the Welsh Government [Grant Number 80761-AU-137 (West)]:\\ 
\includegraphics[width=0.17 \linewidth] {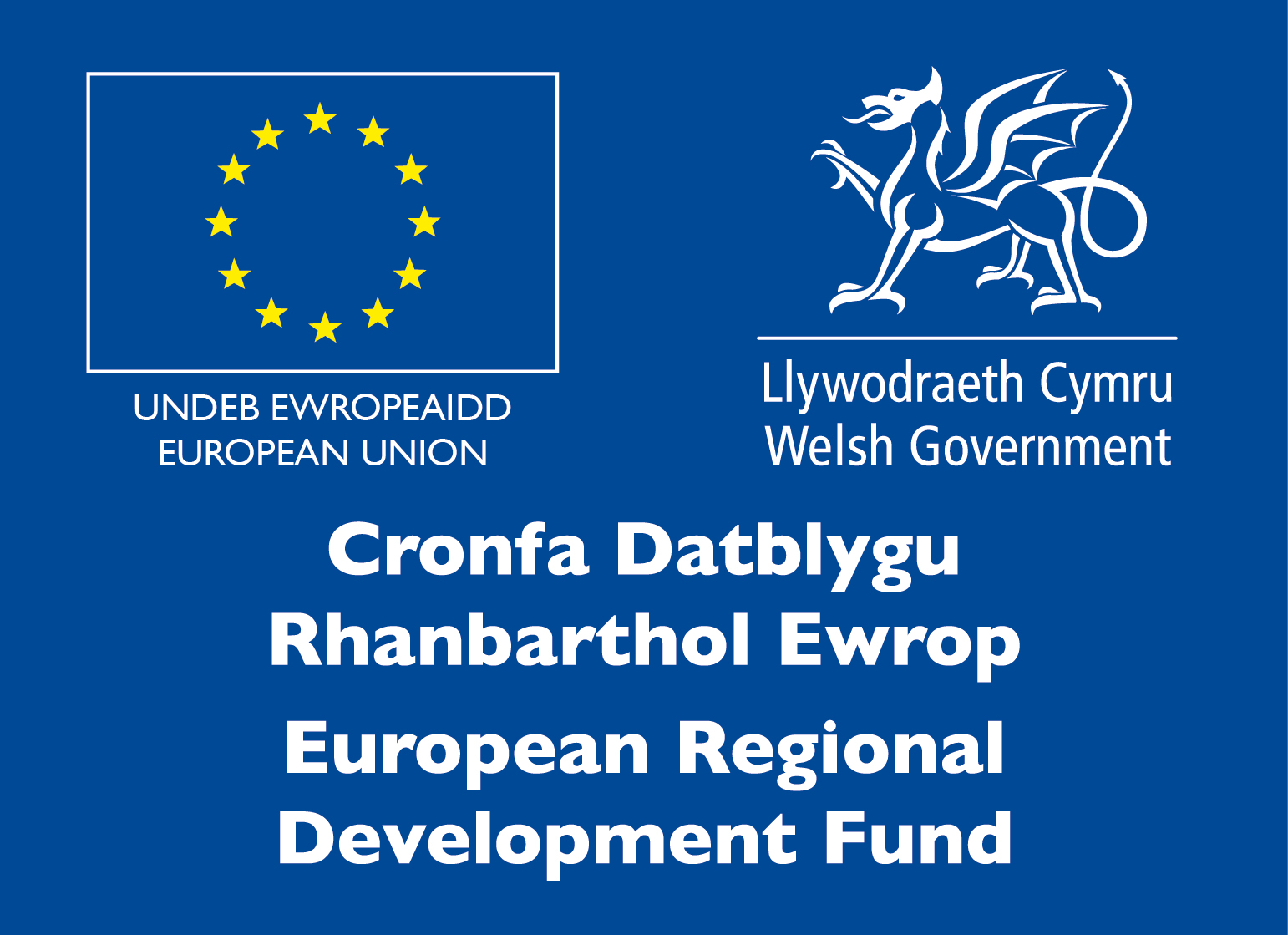}

\clearpage
\bibliographystyle{abbrv}

\end{document}